\documentclass{article}
\usepackage{amssymb, amsmath, amsthm, mathtools, bbm, tikz-cd, stmaryrd, enumerate,cite}

\newcommand{\TT}{{T\oplus T^*}}

\newcommand{\RR}{\mathbb{R}}

\newcommand{\FF}{\mathcal{F}}
\newcommand{\bl}{[\![}
\newcommand{\br}{]\!]}
\newcommand{\bbl}{\bl}
\newcommand{\bbr}{\br}
\newcommand{\brac}{[\ ,\ ]}
\newcommand{\bbrac}{\bbl\ ,\ \bbr}
\newcommand{\bracd}{\bl\ ,\ \br}
\newcommand{\la}{{\langle}}
\newcommand{\ra}{{\rangle}}
\newcommand{\Tp}{{T_+}}
\newcommand{\Tm}{{T_-}}
\newcommand{\xp}{{x_+}}
\newcommand{\yp}{{y_+}}
\newcommand{\zp}{{z_+}}
\newcommand{\PP}{{\mathrm{P}_{\! +}}}
\newcommand{\PPt}{{\mathrm{P}_{\! -}}}
\newcommand{\PPB}{\mathrm{P}^{B}_{\! +}}
\newcommand{\PPtB}{\mathrm{P}^{B}_{\! -}}
\newcommand{\PPBpm}{{\mathrm{P}^{B}_{\! \pm}}}
\newcommand{\PPpm}{{\mathrm{P}_{\! \pm}}}
\newcommand{\id}{{\mathbbm{1}}}

\newcommand{\lc}{\mathring{\n}}
\newcommand{\im}{\mathrm{Im}\,}

\newcommand{\Lie}{\mathcal{L}}
\newcommand{\PS}{\mathcal{P}}
\newcommand{\ap}{\alpha}
\newcommand{\bt}{\beta}
\def\w{\wedge}
\newcommand{\p}{\partial}
\newcommand{\pt}{\tilde{\partial}}
\newcommand{\xt}{{x_-}}
\newcommand{\yt}{{y_-}}
\newcommand{\zt}{{z_-}}
\newcommand{\n}{\nabla}
\newcommand{\rd}{\mathrm{d}}

\newcommand{\se}{\Gamma}

\def\tl{\tilde}

\newtheorem{theorem}{Theorem}[section]
\newtheorem{proposition}[theorem]{Proposition}
\newtheorem{lemma}[theorem]{Lemma}
\newtheorem{corollary}[theorem]{Corollary}

\newtheorem*{theorem*}{Theorem}
\newtheorem*{lemma*}{Lemma}
\newtheorem*{proposition*}{Proposition}

\theoremstyle{definition}
\newtheorem{Def}[theorem]{Definition}

\theoremstyle{definition}
\newtheorem*{notation*}{Notation}
\newtheorem*{definition*}{Definition}

\theoremstyle{remark}

\newtheorem{Ex}[theorem]{Example}

\newenvironment{Rem}[1]{\par\noindent\textbf{{Remark:}}\space#1}{}

\newtheoremstyle{ref}{}{}{\itshape}{}{\bfseries}{.}{.5em}{#1 \thmnote{#3}}
\theoremstyle{ref}
\newtheorem*{reftheorem}{Theorem}
\newtheorem*{refproposition}{Proposition}
\newtheorem*{refdefinition}{Definition}

\input xy

\xyoption{all}

\DeclareMathOperator{\End}{End}

\title{Algebroid Structures on Para-Hermitian Manifolds}
\author{David Svoboda\\ \\
\it{Perimeter Institute For Theoretical Physics}\\
\it{31 Caroline St. N., N2L2Y5, Waterloo ON, Canada}\thanks{e-mail:dsvoboda@perimeterinstitute.ca}\ \footnote{\it This article may be downloaded for personal use only. Any other use requires prior permission of the author and AIP Publishing. This article appeared in \cite{Svoboda:2018rci} and may be found at  https://aip.scitation.org/doi/10.1063/1.5040263}}
\date{}

\begin{document}
\maketitle
\begin{abstract}
We present a global construction of a so-called D-bracket appearing in the physics literature of Double Field Theory (DFT) and show that if certain integrability criteria are satisfied, it can be seen as a sum of two Courant algebroid brackets. In particular, we show that the local picture of the extended space-time used in DFT fits naturally in the geometrical framework of para-Hermitian manifolds and that the data of an (almost) para-Hermitian manifold is sufficient to construct the D-bracket. Moreover, the twists of the bracket appearing in DFT can be interpreted in this framework geometrically as a consequence of certain deformations of the underlying para-Hermitian structure.
\end{abstract}
\newpage
\section{Introduction}
This paper is intended as a more in-depth analysis and an extension of the mathematical framework presented in the paper \cite{freidel2017generalised}, where an interested reader can find more details about the relevant physical background and motivation. We present a construction of a bracket operation called the D-bracket and argue that this bracket is a natural operation on the tangent bundle of an (almost) para-Hermitian manifold, outlining its various potential applications to mathematical problems, especially in deformation theory.

\subsection*{Motivation}
The motivation for the presented construction is largely coming from physics, namely Double Field Theory (DFT){\cite{Tseytlin:1990nb, Tseytlin:1990va, Siegel:1993xq,Siegel:1993th,Hull:2009mi}} (for an overview and more references see \cite{Aldazabal:2013sca,geissbuhler2013exploring}), which considers the following local picture of a so-called \textit{extended space-time}. Similiar constructions of doubled geometry also appear in \textit{Metastring Theory} and \textit{Born Geometry} \cite{freidel2014born,freidel2015metastring}.

Let $\{x^i,\tl{x}_j\}^{i=1\cdots n}_{j=n\cdots 2n}$ be a set of $2n$ local coordinates such that the tangent bundle frame $\{\p_i=\frac{\p}{\p x^i},\pt^j=\frac{\p}{\p \tl{x}_j}\}$ satisfies
\begin{align*}
\eta(\p_i,\pt^j)=\eta(\pt^j,\p_i)=\eta_i^j=\delta_i^j,\quad \eta(\p_i,\p_j)=\eta_{ij}=\eta(\pt^i,\pt^j)=\eta^{ij}=0,
\end{align*}
with respect to a constant metric $\eta$, expressing a certain duality between the coordinates $x^i$ and $\tl{x}_j$. Physical considerations then dictate \cite{Siegel:1993th,Siegel:1993xq,Hull:2009zb} that the tangent bundle of this extended space-time is equipped with a so-called \textbf{D-bracket} instead of the usual Lie bracket of vector fields. The D-bracket is defined with respect to the coordinate frame $\{\p_i,\pt^j\}$ as
\begin{align}\label{eq:D-bracket}
\bl X,Y \br^D=\left(X^I\p_IY^J-Y^I\p_IX^J+\eta_{IL}\eta^{KJ}Y^I\p_KX^L\right)\p_J,
\end{align}
where the capital indices run through $1,\cdots,2n$ so that a $2n$-component vector field
\begin{align}\label{intro-vectorfield}
X=
\begin{pmatrix}
X^i\\
\tl{X}_j
\end{pmatrix},
\end{align}
can be written as $X^I\p_I=X^i\p_i+\tl{X}_j\pt^j$ and $\eta_{IJ}$ is the inverse matrix of $\eta^{IJ}$.

The D-bracket in general does not satisfy the Jacobi identity,
\begin{align*}
\bl X,\bl Y,Z\br^D\br^D-\bl Y,\bl X,Z\br^D\br^D\neq \bl\bl X,Y\br^D,Z\br^D, 
\end{align*}
but it has been observed in \cite{Siegel:1993th,Siegel:1993xq} that the Jacobi identity is satisfied when the bracket is restricted to vector fields obeying a so-called \textbf{section condition}. This means that the coefficient functions $X^i(x,\tl{x})$ and $\tl{X}_j(x,\tl{x})$ of a generic vector field \eqref{intro-vectorfield} must only depend on one coordinate out of each pair $(x^k,\tl{x}_k)$ for every $k=1,\cdots,n$. In other words, the coordinate dependence of the vector field $X$ is restricted to a half-dimensional subspace. Different choices of such subspaces are called {\it polarizations} and should be thought of as restrictions of the extended space-time to an $n$-dimensional, {\it physical} space-time. The different polarizations are then related by \textit{T-duality transformations}, which manifest themselves as an exchange of the individual $x^k\leftrightarrow \tl{x}_k$ coordinates.

One particular polarization is given by setting
\begin{align}\label{eq:polarization}
\pt^k=0, \forall k=1,\cdots,n.
\end{align}
In this polarization, the following property of the D-bracket is observed  -- upon performing a formal exchange
\begin{align}\label{eq:exchange_partial_dx}
\pt^i\rightarrow dx^i,
\end{align}
one recovers a local coordinate expression for the Dorfman bracket:
\begin{align*}
\bl X+\ap,Y+\bt\br=[X,Y]+\Lie_X\bt-\imath_Y\rd\ap,
\end{align*}
where $\ap$ and $\bt$ are one-forms obtained by mapping the last $n$ components of the $2n$-component vector fields $X,Y$ as in \eqref{intro-vectorfield} via the map \eqref{eq:exchange_partial_dx}, i.e. in components we get
\begin{align*}
\ap=\tl{X}_i dx^i,\quad \bt=\tl{Y}_idx^i.
\end{align*}

In various physics applications, the D-bracket is modified by tensorial quantities called \textbf{fluxes} to obtain a so-called  {\it twisted} D-bracket:
\begin{align*}
\bl X,Y\br^{D,\mathcal{F}}=\bl X,Y\br^D+\mathcal{F}(X,Y).
\end{align*}
Certain twists of the D-bracket yield upon imposing the section condition well-known twisted brackets, for example in the polarization \eqref{eq:polarization} the twisted D-bracket reduces to the well-known $H$-twisted Dorfman bracket:
\begin{align}\label{eq:H-dorfman}
\bl X+\ap,Y+\bt\br_H=[X,Y]+\Lie_X\bt-\imath_Y\rd\ap+H(X,Y).
\end{align}
In general, however, the fluxes are given by very complicated expressions (see for example \cite[Eq. 5.16]{geissbuhler2013exploring}) and their unified geometrical interpretation is unclear.

Even though the above local coordinate description is typically sufficient in the physics applications, the question of how the local patches fit in a global geometrical picture is of interest to physicists and has been tackled in both the physics \cite{hohm2013towards,berman2014global,freidel2017generalised} and mathematics \cite{vaisman2012geometry,vaisman2013towards} literature. In this work we employ the approach of \cite{vaisman2012geometry,vaisman2013towards} and \cite{freidel2017generalised}, where the proposed global geometry is that of an (almost) para-Hermitian manifold.

\subsection*{Results}
We show that the above picture can be generalized to the case of a not necessarily flat (almost) para-Hermitian manifold, while the discussion above corresponds to the case when the manifold is flat. The flat limit will still be repeatedly referenced in relation to the usual results known in physics literature.

Let $(\PS,K,\eta)$ be an almost para-Hermitian manifold. In order to define the D-bracket, we introduce a notion of \textbf{adapted connections} (see Definition \ref{def:adapted_conn}), which is a class of connections satisfying certain compatibility conditions with the para-Hermitian data. Given such connection $\n$, the D-bracket is defined in the following way:
\begin{refdefinition}[\ref{def:D-bracket}]
Let $(\PS,\eta,K)$ be an almost para-Hermitian manifold and $\n$ an adapted connection. We define the \textbf{D-bracket} by
\begin{align}
\eta(\bl X,Y\br^D,Z)\coloneqq \eta(\n_{X}Y-\n_{Y}X,Z)+\eta(\n_{Z}X,Y).
\end{align}
\end{refdefinition}
Similar formula was used to define a bracket operation by Vaisman \cite{vaisman2012geometry} in the context of so called metric algebroids. Denote now the eigenbundles of the almost para-Hermitian structure $K$ by $T_\pm$. There exist natural vector bundle isomorphisms (see Definition \ref{eta_iso}):
\begin{align*}
\begin{aligned}
\rho_\pm:\ T\PS=&\Tp\oplus \Tm\rightarrow T_\pm\oplus T_\pm^*\\
X=&x_++\xt \mapsto x_\pm+\eta(x_\mp).
\end{aligned}
\end{align*}
When either of the distributions $T_\pm$ is integrable, we denote the corresponding integral foliations $\FF_\pm$. The isomorphisms then induce Courant algebroid brackets $\bracd_\pm$ on the tangent bundle $T\PS$:
\begin{align*}
\bbl X,Y\bbr_\pm\coloneqq\rho_\pm^{-1}\bbl \rho_\pm X,\rho_\pm Y\bbr_{\FF_\pm},
\end{align*}
where $\bracd_{\FF_\pm}$ denotes the standard Dorfman brackets on $(\TT)\FF_\pm$.
To relate the D-bracket defined in this way to the Dorfman bracket, we define the \textbf{$\PPpm$-projected brackets} associated to an arbitrary connection by
\begin{align*}
\eta(\bl X,Y\br^\n_\pm,Z)\coloneqq\eta(\n_{\PPpm X}Y-\n_{\PPpm Y}X,Z)+\eta(\n_{\PPpm Z}X,Y),
\end{align*}
where $\PPpm$ are projections onto $T_\pm$. These projected brackets then coincide with the Courant algebroid brackets $\bracd_{\FF_\pm}$ iff the connection $\n$ is appropriately adapted to the foliated manifold (see Definitions \ref{def:pnhermitian} and \ref{def:adapted_conn} for the notions of a p(n)-para-Hermitian manifold and p(n)-adapted connections):
\begin{reftheorem}[\ref{prop:metric-courant-TP}]
Let $(\PS,\eta,K)$ be a p(n)-para-Hermitian manifold and $\n$ a connection on $T\PS$. The $\PPpm$-projected bracket associated to $\n$ is equal to the Courant algebroid bracket $\bracd_+$,
\begin{align*}
\bracd_+^\n=\bracd_+,
\end{align*}
if and only if $\n$ is a p(n)-adapted connection.
\end{reftheorem}
The relationship between the Dorfman bracket and the ``doubled" D-bracket is therefore within our framework realized by projecting the vector fields in the arguments of the connections defining the brackets as opposed to restricting the coordinate dependence of the vector fields. It will become clear that this way of constructing the D-bracket allows for relaxation of the section condition approach taken in physics.

We now show how to recover usual results in special cases when $K$ is fully integrable and when the underlying pseudo-Riemannian manifold is flat (i.e. the Levi-Civita connection of $\eta$ has vanishing curvature tensor).
When $K$ is integrable, the D-bracket can be written as a sum of the brackets $\bracd_\pm$,
\begin{align*}
\bbl X,Y\bbr^D=\bbl X,Y\bbr_++\bbl X,Y\bbr_-,
\end{align*}
which recovers the definition used in \cite{vaisman2013towards}. Moreover, the integrability of $K$ implies that every local patch $U\subset \PS$ can be equipped with a set of local coordinates $\{x^i,\tl{x}_j\}^{i=1\cdots n}_{j=n+1\cdots 2n}$ called \textbf{adapted coordinates}, such that the the eigenbundles $T_+$ and $T_-$ are spanned by $\{\p_i=\frac{\p}{\p x^i}\}$ and $\{\pt^i=\frac{\p}{\p \tl{x}_i}\}$, respectively. If the manifold is flat, the adapted coordinates can be chosen such that $\eta(\p_i,\pt^j)=\delta_i^j$, $\eta(\p_i,\p_j)=\eta(\pt^i,\pt^j)=0$ and the maps $\rho_\pm$ realize the mapping \eqref{eq:exchange_partial_dx} used in physics:
\begin{align*}
\pt^i \overset{\rho_+}{\longrightarrow} dx^i, \quad \p_i\overset{\rho_-}{\longrightarrow} d\tl{x}_i.
\end{align*}
In such flat case, we also recover the usual local expression \eqref{eq:D-bracket} for the D-bracket.

We continue the discussion by introducing a natural endomorphism on the tangent bundle $T\PS$, which preserves the eigenbundle $T_-$ but changes $T_+$; this can be geometrically viewed as a deformation of the splitting of $T\PS$, which \textit{shears} $T_+$ in the direction of $T_-$. It turns out that such maps provide the right framework for understanding the fluxes and twisted brackets known in physics literature, in our setting of para-Hermitian manifolds. These maps are also closely related to shearing transformations of generalized geometry, as we note in Section \ref{sec:GG-relationship}.

We start by considering a skew map
\begin{align*}
B:T_+\rightarrow T_-,\quad \eta(BX,Y)=-\eta(X,BY).
\end{align*}
This is equivalent to a choice of a two-form $b_+\coloneqq\eta B\in \se(\Lambda^2 T_+^*)$ or a bi-vector $b_-\coloneqq B\eta^{-1}\in\se(\Lambda^2 T_-)$. $B$ then defines an endomorphism of the tangent bundle, denoted by $e^B$ and called the \textbf{$B$-transformation},  given in the adapted frame of $T\PS$ by
\begin{align*}
e^B=
\begin{pmatrix}
\id & 0 \\
B & \id
\end{pmatrix}.
\end{align*}
Such an endomorphism on the tangent bundle of a para-K\"ahler manifold induces a new almost para-Hermitian structure $K_B$:
\begin{align*}
K_B\coloneqq e^{-B}K e^B=
\begin{pmatrix}
\id & 0 \\
2B & -\id
\end{pmatrix},
\end{align*}
called the \textbf{B-transformation} of $K$. We show that the eigenbundles of any almost para-Hermitian structure are involutive with respect to its corresponding D-bracket; it is therefore natural to ask under which conditions the eigenbundles of $K_B$ are involutive under the D-bracket of $K$:
\begin{refproposition}[\ref{prop:weak-int-maurer-cartan}]
Let $(K_B,\eta)$ be a B-transformation of a para-Hermitian structure $(K,\eta)$ and let $\bracd^D$ be the D-bracket associated to $(K,\eta)$. The eigenbundles of $K_B$ are involutive under the D-bracket if and only if
\begin{align}\label{eq:intro-maurer-cartan}
\rd_+b_++(\Lambda^3\eta)[b_-,b_-]_-=0,
\end{align}
where $\rd_+$ is the Lie algebroid differential of $T_+$ and $\brac_-$ is the Schouten bracket of $T_-$.
\end{refproposition}
The equation \eqref{eq:intro-maurer-cartan} can therefore be seen as a natural condition for \textbf{compatibility} of $K_B$ with $K$ (see Definition \ref{def:compatible}).

Finally, to interpret the fluxes appearing in the physics literature, we restrict our attention to para-K\"ahler manifolds, which is a special class of para-Hermitian manifolds for which the fundamental form $\omega=\eta K$ is closed. We find that the D-bracket associated to $K_B$ is a twisted D-bracket:
\begin{refproposition}[\ref{prop:twistedDbrac}]
Let $K_B$ be a B-transformation of a para-K\"ahler structure $(\PS,\eta,K)$. The D-bracket associated to $K_B$ is given by
\begin{align*}
\eta(\bl X,Y\br^{D,B},Z)=\eta(\bl X,Y\br^D ,Z)-(\rd b_+)(X,Y,Z).
\end{align*}
where $\bl\ ,\ \br^D$ denotes the D-bracket of $K$.
\end{refproposition}

The different components of $\rd b_+$ then give exactly the $H,Q$ and $R$ fluxes well-known in physics \cite{Shelton:2005cf,andriot2012non,blumenhagen2012bianchi,halmagyi2009non}. Additionally, the sum of the $H$ and $R$ fluxes appears on the left hand side of equation \eqref{eq:intro-maurer-cartan}. We therefore find that the twists of the D-bracket can be interpreted geometrically as a consequence of a deformation of the para-K\"ahler structure $K$ into an almost para-Hermitian structure $K_B$. The corresponding fluxes then give the obstruction to the compatibility of $K_B$ with $K$.

The paper is organised as follows. Section 2 gives an introduction to Lie and Courant algebroids. Section 3 includes a brief introduction to para-Hermitian geometry followed by a definition of a pair of Courant algebroids on the tangent bundle of the para-Hermitian manifold. In section 4 we introduce a formalism which uses adapted connections to define the D-bracket for any almost para-Hermitian manifold and finally in section 5 we discuss a certain deformation problem of para-K\"ahler manifolds, how it relates to the D-bracket and its relationship to fluxes. We conclude the paper in section 6 by outlining possible future research directions and implications to physics.

It has been brought to the author's attention that an e-print of a work on closely related subjects \cite{Chatzistavrakidis:2018ztm} has been posted very recently. In this work, the authors propose a model for the DFT extended spacetime as the total space of a cotangent bundle $T^*M$ of a Riemannian manifold $(M,g)$, which is in fact a canonical example of an (almost) para-Hermitian manifold\footnote{\cite{dructua2013riemannian} presents a construction of the almost para-Hermitian structure on $T^*M$ and discusses conditions under which this structure is integrable.}. The authors further also observe the existence of the pair of the canonical Courant algebroids discussed at the end of section \ref{couralgds}.
\section{Algebroids}
In this section, we recall the definitions of Lie and Courant algebroids. At the end of the section, we provide an important example motivating the discussion carried out in section \ref{sec:adapted}. For more details and examples of Lie and Courant algebroids, consult for example \cite{weinsteingeomodels,Roytenberg} for the former and \cite{Liu:1995lsa,Roytenberg} for the latter.
\subsection{Lie Algebroids}
\begin{Def}
Let $E\rightarrow M$ be a vector bundle. A \textbf{Lie algebroid} is a triple $(E, \brac_E,a)$, where $\brac_E$ is a skew bracket operation on $\se(E)$ and $a$ is a bundle map $E\rightarrow TM$ called the \textbf{anchor}, satisfying
\begin{align}
a([ X,Y]_E)&=[a(X),a(Y)].\label{anchoring}
\end{align}
Furthermore, the bracket is a derivation on $C^\infty(M)$,
\begin{align}
[ X,fY]_E&=a(X)[f]Y+f[ X,Y]_E.\label{liealg_leibniz}
\end{align}
\end{Def}
\begin{Ex}
For any manifold $M$, the \textbf{tangent}, \textbf{canonical} or \textbf{standard Lie algebroid} $(TM,\brac,\id_{TM})$ is given by the Lie bracket and identity anchor.
\end{Ex}
It is sometimes useful to think of Lie algebroids as generalisations of the tangent Lie algebroid. In particular, Lie algebroids provide a way of thinking of $\se(E)$ as derivations on $M$. We get another example of a Lie algebroid arising from considering an integrable distribution $\mathcal{D}\subset TM$:

\begin{Ex}\label{ex:alg2}
Let $\mathcal{D}\subset TM$ be an integrable distribution integrating to a foliation $\FF=\sqcup_i \mathcal{F}_i$, where $\FF_i$ are the leafs of $\FF$. Then $\mathcal{D}\rightarrow \FF$ is the tangent Lie algebroid of $\FF$. Considering $\mathcal{D}$ as a subbundle of $TM$, the Lie bracket on $\mathcal{D}=T\FF$ is the restriction of the Lie bracket on $TM$ to $\mathcal{D}$. 
\end{Ex}

\subsubsection{The Gerstenhaber algebra of a Lie algebroid}
Given a Lie algebroid $E$, we can uniquely extend the Lie algebroid bracket to sections of $\Lambda^\bullet(E)$. Such a bracket is called the (generalized) Schouten bracket, satisfies a graded version of Jacobi identity and gives $\Lambda^\bullet(E)$ the structure of a \textbf{Gerstenhaber algebra}. The Schouten bracket satisfies the following:
\begin{lemma}\label{lem:schouten}
Let $\bt$ be a bi-vector on a manifold $M$ and let $\n$ be a torsionless connection. Then the Schouten bracket of $\bt$ with itself is given by\footnote{Depending on conventions of exterior products, sometimes a normalization constant $\frac{1}{2}$ is added in front of $[\bt,\bt]$.}
\begin{align*}
[\bt,\bt](\lambda,\mu,\nu)=\mkern-18mu\sum_{Cycl.\ \lambda,\mu,\nu}\mkern-18mu \n_{\bt(\lambda)}\bt(\mu,\nu),
\end{align*}
where $\lambda,\mu,\nu \in \Omega^1(M)$.
\end{lemma}
\begin{proof}
Writing $\bt=\bt^{ij}\p_i\w\p_j$, we compute the Schouten bracket:
\begin{align*}
[\bt,\bt]=\left(\bt^{il}\p_l(\bt^{jk})+\bt^{jl}\p_l(\bt^{ki})+\bt^{il}\p_l(\bt^{jk})+\bt^{kl}\p_l(\bt^{ij}))\right)\p_i\w\p_j\w\p_k.
\end{align*}
All the terms can be brought to the same form by cyclic permutation and contracting in the one-forms $\lambda,\mu,\nu$ yields (omitting the normalization factors)
\begin{align*}
[\bt,\bt](\lambda,\mu,\nu)=\bt^{il}\p_l(\bt^{jk})\left(\sum_{Cycl.\ ijk}\mkern-18mu \lambda_i\mu_j\nu_k\right),
\end{align*}
which can be rewritten using a torsionless connection into the desired form.
\end{proof}

The data of the Gerstenhaber algebra on $\Lambda^\bullet(E)$ can be equivalently given by a Lie algebroid \textbf{exterior derivative} $\rd_E$, which is a degree $1$ derivation on $\Lambda^\bullet(E^*)$. It is uniquely determined by its action on $C^\infty(M)$ and $\se(E^*)$:
\begin{align}\label{d_oneform}
\begin{aligned}
(\rd_Ef)(X)&=a(X)[f]\\
(\rd_E\xi)(X,Y)&=a(X)\xi(Y)-a(Y)\xi(X)-\xi([ X,Y]_E),
\end{aligned}
\end{align}
i.e. $\rd_E: C^\infty(M)\rightarrow E^*$ is given by $a^*\rd$, where $\rd$ is the de-Rham differential on $M$.

The Lie algebroid structure $(E, \brac,a)$ also comes with a natural (generalized) Lie derivative on $E^*$ along sections of $E$, given by \begin{align}\label{eq:lieder}
\Lie^E_X\xi\coloneqq \rd_E\imath_X\xi+\imath_X\rd_E\xi,\quad X\in \se(E),\ \xi \in \se(E^*)
\end{align}
where $\imath_X$ is the contraction induced by the duality pairing between $E$ and $E^*$. Lastly, the bundle $E\oplus E^*$ is equipped with a natural inner product $\langle\ ,\ \rangle$, given by
\begin{align}\label{eq:pairing}
\langle X+\ap,Y+\bt\rangle = \imath_X\bt+\imath_Y\ap.
\end{align}

\subsection{Courant Algebroids}
\begin{Def}\label{def:courant}
Let $E\rightarrow M$ be a vector bundle. A \textbf{Courant algebroid} is a quadruple $(E,a,\langle\ , \ \rangle_E,\bbl\ ,\ \bbr)$, where $a$ is bundle map $E\rightarrow TM$ called the \textbf{anchor}, $\langle\ , \ \rangle_E:\se(E)\times\se(E)\rightarrow C^\infty(M)$ is a non-degenerate symmetric pairing and $\bbl\ ,\ \bbr:\se(E)\times\se(E)\rightarrow \se(E)$ is a bracket operation called the \textbf{Dorfman bracket}\footnote{The Dorfman bracket is sometimes called the Dorfman derivative or generalized Lie derivative}, such that
\begin{enumerate}
\item $a(X)\langle Y,Z\rangle_E=\la\bbl X,Y\bbr,Z\ra_E+\la Y,\bbl X,Z\bbr\ra_E$
\item $\langle\bbl X,X\bbr,Y\rangle_E=\frac{1}{2}a(Y)\langle X,X\rangle_E$
\item $\bbl X \bbl Y,Z\bbr\bbr=\bl\bl X,Y\br Z\br+\bl Y,\bl X,Z\br\br$,
\end{enumerate}
\end{Def}

\begin{Rem}
This structure is sometimes called a \textbf{Leibniz algebroid}, reflecting the fact that the operation $\bracd$ satisfies the property $(3)$, often called the Leibniz property or the Jacobi identity.
\end{Rem}
\begin{Rem}
A Courant algebroid can be also defined using the \textbf{Courant bracket}, which is a skew version of the Dorfman bracket, $
\bbl X,Y \bbr_{Cour.} \coloneqq \frac{1}{2}(\bbl X,Y\bbr-\bbl Y,X\bbr)$. This bracket, however, fails to satisfy the Jacobi identity. In the following, we use the Dorfman bracket, but the entire discussion can be carried out using the Courant bracket as well.
\end{Rem}
The following example shows how a particular type of Courant algebroid arises from a Lie algebroid:
\begin{Ex}
Let $(E,\brac_E,a)$ be a Lie algebroid. Then $(E\oplus E^*,a\oplus 0, \langle\ , \ \rangle, \bbl\ ,\ \bbr)$, where $\langle\ , \ \rangle$ is given by \eqref{eq:pairing} and $\bbl\ ,\ \bbr$ is defined as 
\begin{align}\label{dorfbrac}
\bbl X+\ap,Y+\bt\bbr=&[X,Y]_E+\Lie^E_X\bt-\Lie^E_Y\ap+\rd_E\langle \ap(Y)\rangle,
\end{align}
is a Courant algebroid. For the tangent Lie algebroid $(TM,\brac,\id_{TM})$, this Courant algebroid is called the \textbf{standard Courant algebroid}.
\end{Ex}

\begin{notation*}
In the following we denote $TM\oplus T^*M \coloneqq (\TT)M$ and a section $e$ of this bundle will be written as $e=X+\ap$, i.e. the vector field part is denoted by a capital Roman letter while the one-form part is denoted by a lower-case Greek letter. The natural projection onto $TM$ given by $X+\ap \mapsto X$ will be denoted by $\pi_T$.
\end{notation*}

We end our discussion with an important observation about the standard Courant algebroid:
\begin{proposition}\label{ex:dorfman_connec}
Let $((\TT)M,\pi_T,\la \ ,\ \ra,\bbrac)$ be the standard Courant algebroid on a manifold $M$ and $\n$ any torsionless connection on $TM$. The Dorfman bracket $\bbrac$ is given by
\begin{align}
\la \bbl{e_1},e_2\bbr,e_3\ra=\la \n_{\pi_T(e_1)}e_2-\n_{\pi_T(e_2)}e_1,e_3\ra+\la\n_{\pi_T(e_3)}e_1,e_2\ra,
\end{align}
for $e_i,\ i=1,\cdots,3$ sections of $\TT(M)$.
\end{proposition}
\begin{proof}
To check formula \eqref{eq:ex_dorf_conn}, we write $e_i$ as $e_1=X+\ap$, $e_2=Y+\bt$ and $e_3=Z+\gamma$:
\begin{align*}
\langle \bbl X\!+\!\ap,Y\!+\!\bt\bbr,Z\!+\!\gamma\rangle&=\langle\n_X(Y\!+\!\bt)-\n_Y(X\!+\!\ap),Z\!+\!\gamma\rangle\!+\!\langle \n_Z(X\!+\!\ap),Y\!+\!\bt\rangle\\
&=\langle [X,Y],\gamma\rangle+\langle\n_X\bt-\n_Y\ap,Z\rangle+\langle\n_ZX,\bt\rangle+\langle\n_Z\ap,Y\rangle\\
&=\langle [X,Y],\gamma\rangle+\langle \Lie_X\bt-\Lie_Y\ap+\rd[\ap(Y)],Z\rangle,
\end{align*}
where we used $\langle\Lie_X\bt,Z\rangle=X\langle \bt,Z\rangle-\langle \bt,[X,Z]\rangle$.
\end{proof}

\section{Para-Hermitian Geometry}
We now recall some basic results of para-complex and para-Hermitian geometry. For more details see \cite{Cruceanu} and references therein. Canonical examples of (almost) para-Hermitian manifolds are given by the total space of the tangent and cotangent bundle of a manifold \cite{dructua2013riemannian,vilcu2011hyperhermitian}. Further examples of para-Hermitian and para-K\"ahler manifolds can be found for example in \cite{vaisman2013towards,bilagrangian}, and a classification of almost para-Hermitian manifolds is given in \cite{gadea1991classification}. For a discussion on the existence of para-Hermitian vector bundles, see \cite{bejan1993existence}.
\begin{Def}
Let $E\rightarrow M$ be a rank $2n$ vector bundle. $K\in \End(E)$ is called a \textbf{para-complex structure} on $E$ if $K^2=\id$ and the $\pm 1$ eigenbundles $E_\pm \subset E$ of $K$ have the same rank. A symmetric, non-degenerate pairing $\eta \in \se(E^*\otimes E^*)$ is called para-Hermitian if $\eta(K\cdot,K\cdot)=-\eta$. The pair $(K,\eta)$ is then called a \textbf{para-Hermitian structure} on  $E$.
\end{Def}
\begin{Rem}
The above definition forces $\eta$ to be of signature $(n,n)$ and the eigenbundles $E_\pm$ are necessarily maximally isotropic: $\eta(E_\pm,E_\pm)=0$ and $\text{Rank}(E)=n$.
\end{Rem}
\begin{Ex}\label{ex:TT}
For any manifold $M$, the \textbf{generalized tangent bundle} $(\TT)M$ is equipped with a natural para-Hermitian structure given by
\begin{align*}
K=
\begin{pmatrix}
\id_T & 0 \\
0 & -\id_{T^*}
\end{pmatrix}, \quad
\eta(X+\ap,Y+\bt)=\ap(Y)+\bt(X).
\end{align*}
\end{Ex}
\begin{Def}
An \textbf{almost para-Hermitian manifold} $\PS$ is a manifold whose tangent bundle carries a para-Hermitian structure. 
\end{Def}
\begin{Rem}
If $(\PS,K,\eta)$ is an almost para-Hermitian manifold, then $K$ is an almost para-complex structure on $T\PS$ and $\eta$ is a pseudo-Riemannian metric of split signature, forcing $\PS$ to be even-dimensional. Furthermore, the tensor $\omega\coloneqq \eta K$ is skew
\begin{align*}
\omega(X,Y)=\eta(KX,Y)=-\eta(X,KY)=-\omega(Y,X),
\end{align*}
and nondegenerate (because $\eta$ is nondegenerate). Therefore, $\omega$ is an almost symplectic form, sometimes called the \textbf{fundamental form}.  It follows that the $\pm 1$ eigenbundles $T_\pm\subset T\PS$ of $K$ are isotropic with respect to both $\eta$ and $\omega$.
\end{Rem}
\begin{Def}
Let $(\PS,K,\eta)$ be an almost para-Hermitian manifold. The \textbf{para-Hermitian projections} $\PPpm$ are defined by
\begin{align}\label{eq:projections}
\PPpm\coloneqq\frac{1}{2}(\id\pm K):\ T\PS\rightarrow T_\pm
\end{align}
The projections satisfy $\PPpm^{\! 2}=\PPpm$, $\im(\PPpm)=T_\pm$ and $\PP+\PPt=\id$.
\end{Def}
\begin{notation*}
The data $(\PS,K,\eta)$, $(\PS,\eta,\omega)$ and $(\PS,K,\omega)$ on an almost para-Hermitian manifold are equivalent and so we may use the different triples interchangeably. We will denote the splitting of vector fields corresponding to the $\pm 1$ eigenbundles of $K$ by small letters with $\pm$ subscripts, e.g. $X=x_++x_-$, where $x_+=\PP(X)$ and $\xt=\PPt(X)$.
\end{notation*}

The metric $\eta$ induces the following isomorphism.
\begin{lemma}
Let $(\PS,\eta,K)$ be an almost para-Hermitian manifold. Let also $T_\pm^*$ be the dual vector bundles to $T_\pm$. The bundles $T_\pm$ and $T_\mp^*$ are then isomorphic,
\begin{align*}
T_\pm \ {\simeq}\ T_\mp^*,
\end{align*}
via contraction with $\eta$:
\begin{align*}
\eta: T_\pm &\rightarrow T^*_\mp\\
x_\pm &\mapsto \eta(x_\pm),
\end{align*}
where we used the shorthand notation $\eta(x_\pm)\coloneqq\eta(x_\pm,\cdot)$.
\end{lemma}
\begin{proof}
Both $\Tp$ and $\Tm$ are isotropic with respect to $\eta$, therefore $\eta(x_\pm,\cdot)\subset T_\mp^*$. Because $\eta$ is non-degenerate, this is an isomorphism and $\eta(T_\pm)=T^*_\mp$.
%
\end{proof}
This leads to the following vector bundle isomorphisms, central to our construction \cite{freidel2017generalised}:
\begin{Def}\label{eta_iso}
Let $(\PS,\eta,\omega)$ be an almost para-Hermitian manifold. Define the maps $\rho_\pm$ by
\begin{align}
\begin{aligned}
\rho_\pm:\ T\PS=&\Tp\oplus \Tm\rightarrow T_\pm\oplus T_\pm^*\\
X=&x_++\xt \mapsto x_\pm+\eta(x_\mp).
\end{aligned}
\end{align}
\end{Def}
If $T_+$ is Frobenius integrable, $T_+\oplus T_+^*\simeq (\TT)\FF_+$ and so $\rho_+$ maps
\begin{align*}
\rho_+: T\PS\rightarrow (\TT)\FF_+.
\end{align*}
Similarly, if $T_-$ is Frobenius integrable, then
\begin{align*}
\rho_-:T\PS\rightarrow (\TT)\FF_-.
\end{align*}

\subsection{Integrability}
The integrability of $K$ is similarly to the complex case governed by the \textbf{Nijenhuis tensor}
\begin{align}\label{eq:nijenhuis}
\begin{aligned}
N_K(X,Y)&\coloneqq\frac{1}{4}\left( [X,Y]+[KX,KY]-K([KX,Y]+[X,KY])\right)\\
&=\frac{1}{4}\left( (\n_{KX}K)Y+(\n_XK)KY-(\n_{KY}K)X-(\n_YK)KX\right)\\
&=\PP[\PPt X,\PPt Y]+\PPt[\PP X,\PP Y]),
\end{aligned}
\end{align}
where $\n$ is any torsionless connection and $X,Y,Z$ are any vector fields of $T\PS$. We will also employ the following notation:
\begin{align*}
N_K(X,Y,Z)\coloneqq\eta(N_K(X,Y),Z),\quad N_\pm(X,Y,Z)\coloneqq N_K(\PPpm X,\PPpm Y,\PPpm Z).
\end{align*}
It follows that
\begin{align*}
N_K(X,Y,Z)=N_+(X,Y,Z)+N_-(X,Y,Z).
\end{align*}

We say that $K$ is integrable if $N_K=0$. From \eqref{eq:nijenhuis} it is apparent that $K$ is integrable if and only if both $\Tp$ and $\Tm$ are simultaneously Frobenius integrable, i.e. $\brac$-involutive distributions in $T\PS$. However, unlike in complex geometry, the integrabilities of $\Tp$ and $\Tm$ are independent: $T_+$ can be integrable while $T_-$ is not and vice versa. This gives rise to the notion of ``half-integrability'':
\begin{Def}\label{def:pnhermitian}
Let $(\PS,K,\eta)$ be an almost para-Hermitian manifold with $T_\pm$ the $\pm 1$ eigenbundles of $K$. If $\Tp$ (resp., $\Tm$) is an integrable distribution, we say $(\PS,K,\eta)$ is a \textbf{$p$-para-Hermitian} (\textbf{$n$-para-Hermitian}) manifold. If $(\PS,K,\eta)$ is both $p$-para-Hermitian and $n$-para-Hermitian manifold, we simply call it \textbf{para-Hermitian}.
\end{Def}

\subsection{Type Decomposition}
The splitting of the tangent bundle of any para-complex manifold gives rise to a decomposition of tensors analogous to the $(p,q)$-decomposition in complex geometry. Denote $\Lambda^{(+k,-0)}(T^*\mathcal{P})\coloneqq \Lambda^k(\Tp^*)$ and $\Lambda^{(+0,-k)}(T^*\mathcal{P})\coloneqq\Lambda^k(\Tm^*)$. The splitting is then
\begin{align}\label{eq_plusminus_decomp}
\Lambda^k (T^*\mathcal{P})=\bigoplus_{k=m+n}\Lambda^{(+m,-n)}(T^*\mathcal{P}),
\end{align}
with corresponding sections denoted as $\Omega^{(+m,-n)}(\mathcal{P})$. We note that the fundamental form of a para-Hermitian manifold is a $(+1,-1)$ form, since both $T_\pm$ are Lagrangian with respect to $\omega$.

\subsection{Classes of Para-Hermitian Manifolds}
We now introduce various important classes of (almost) para-Hermitian manifolds. Let $\lc$ be the Levi-Civita connection of the pseudo-Riemannian metric $\eta$. We define the tensor
\begin{align*}
\Phi(X,Y,Z)\coloneqq \eta((\lc_XK)Y,Z)=\lc_X\omega(Y,Z),
\end{align*}
which has the following property:
\begin{lemma}\label{lem:Phi}
The tensor $\Phi$ satisfies
\begin{align*}
\Phi(X,KY,KZ)=\Phi(X,Y,Z),\quad \Phi(X,\PP Y,\PPt Z)=\Phi(X,\PPt Y,\PP Z)=0,
\end{align*}
for any vector fields $X,Y,Z$. Equivalently,
\begin{align*}
\lc_X\omega(KX,KY)=\lc_X\omega(Y,Z),\quad \lc_X\omega(\PP Y,\PPt Z)=\lc_X\omega(\PPt Y,\PP Z)=0.
\end{align*}
\end{lemma}
\begin{proof}
Because $K^2=\id$, $\lc K$ and $K$ anticommute and so
\begin{align*}
\Phi(X,KY,KZ)&=\eta((\lc_XK)KY,KZ)=-\eta(K(\lc_XK)Y,KZ)=\eta(\lc_XK)Y,Z)\\
&=\Phi(X,Y,Z).
\end{align*}
This implies $\Phi(X,\PP Y,\PPt Z)=-\Phi(X,\PP Y,\PPt Z)=0$, hence the second equality.
\end{proof}
$\Phi$ can also be used to express the Nijenhuis tensor:
\begin{align}
N_\pm(X,Y,Z)&=\frac{1}{2}\left[\Phi(\PPpm X,\PPpm Y,\PPpm Z)-\Phi(\PPpm Y,\PPpm X,\PPpm Z)\right]
\end{align}

We now introduce the following nomenclature \cite{Ivanov}:
\begin{Def}
Let $(\PS,\eta,\omega)$ be an almost para-Hermitian manifold.

\begin{itemize}
\item If $\Phi(X,Y,Z)+\Phi(Y,X,Z)=0$, or equivalently if $\Phi(X,Y,Z)$ is fully skew, we call $(\PS,\eta,\omega)$ a \textbf{nearly para-K\"ahler manifold}.
\item If $\rd\omega=0$, we call $(\PS,\eta,\omega)$ an \textbf{almost para-K\"ahler manifold}, and simply a \textbf{para-K\"ahler manifold} if $N_K=0$.
\end{itemize} 
\end{Def}

\begin{Rem}
A para-K\"ahler manifold $(\PS,\eta,\omega)$ can be seen as a symplectic manifold with a preferred choice of Lagrangian distributions $T_\pm$, the unique Lagrangians of $\omega$ isotropic with respect to $\eta$. Such symplectic manifolds are called bi-Lagrangian. For more details, see \cite{bilagrangian,homogeneous-parakahler}.
\end{Rem}

We have the following series of properties and relationships between the special cases of para-Hermitian manifolds, starting from a statement analogous to Hermitian/K\"ahler geometry.
\begin{lemma}\label{lem:herm-kahl}
Let $(\PS,K,\eta)$ be an almost para-Hermitian manifold. Then $(\PS,K,\eta)$ is para-K\"ahler if and only if $\lc K=0$ (or equivalently $\lc \omega=0$), where $\lc$ is the Levi-Civita connection of $\eta$.
\end{lemma}
\begin{proof}
The proof is similar to the proof of the analogous statement in complex geometry. See for example \cite[Theorem 5.5]{Moroianu}.
\end{proof}

Of course, any almost para-K\"ahler manifold is nearly para-K\"ahler. Furthermore, we have the following
\begin{lemma}
The Nijenhuis tensor of a nearly para-K\"ahler manifold is fully skew.
\end{lemma}
\begin{proof}
\begin{align*}
N_\pm(X,Y,Z)=\eta([\PPpm X,\PPpm Y],\PPpm Z)&=\eta((\lc_{\PPpm X}\PPpm) Y-(\lc_{\PPpm Y}\PPpm) X,\PPpm Z)\\
&=\pm \Phi(\PP X,\PP Y,\PP Z),
\end{align*}
therefore $N$ is fully skew.
\end{proof}

%

The following lemma describes the relationship between integrability of $K$ and the $(+3,-0)$ and $(+0,-3)$ parts of $\rd \omega$.

\begin{lemma}\label{lem:30integr}
Let $(\PS,K,\omega)$ be an almost para-Hermitian manifold. Then the following formulas hold:
\begin{align*}
(\rd\omega)^{+3,-0}(X,Y,Z)&=\mkern-18mu\sum_{Cycl.\ X,Y,Z}\mkern-18muN_+(X,Y,Z)\\
(\rd\omega)^{+0,-3}(X,Y,Z)&=-\mkern-18mu\sum_{Cycl.\ X,Y,Z}\mkern-18mu N_-(X,Y,Z).
\end{align*}
Therefore, $T_+$ (resp., $T_-$) being integrable is a sufficient condition for $\rd\omega^{(+3,-0)}$ (resp., $\rd\omega^{(+0,-3)}$) to vanish. If $(\PS,K,\omega)$ is nearly para-K\"ahler, then it is also a necessary condition.
\end{lemma}
\begin{proof}
We make use of the Cartan formula for the exterior derivative of a two-form:
\begin{align*}
\rd \omega (X,Y,Z)=\mkern-18mu\sum_{Cycl.\ X,Y,Z}\mkern-18mu X[\omega(Y,Z)]-\omega([X,Y],Z).
\end{align*}
Because $\rd \omega^{+3,-0}(X,Y,Z)=\rd\omega(\xp,\yp,\zp)$,
\begin{align*}
(\rd\omega)^{+3,-0}(X,Y,Z)=\mkern-18mu\sum_{Cycl.\ \xp,\yp,\zp}\mkern-18mu\eta([\xp,\yp],\zp)=\mkern-18mu\sum_{Cycl.\ X,Y,Z}\mkern-18mu N_+(X,Y,Z),
\end{align*}
where we used $\omega=\eta K$. If $(\PS,\eta,K)$ is nearly para-K\"ahler, then $N$ is fully skew and therefore $\rd \omega^{+3,-0}(X,Y,Z)=N_+(X,Y,Z)$.
\end{proof}

\subsection{The Courant Algebroids of Para-Hermitian Foliations}\label{couralgds}
We will now explore the tangent Lie algebroid on $T_+$ of a $p$-para-Hermitian manifold (see Definition \ref{def:pnhermitian}) and its the corresponding standard Courant algebroid. We then transport this courant algebroid to the tangent bundle $T\PS$, using the isomorphism $\rho_+$. This construction was presented in slightly different language in \cite{vaisman2013towards}; here we recall this construction and fill in various details important for later discussion.

Let $(\PS,\eta,\omega)$ be a $p$-para-Hermitian manifold (analogous discussion holds if $\PS$ is an $n$-para-Hermitian)and denote the foliation corresponding to $\Tp$ by $\mathcal{F}_+$. As described in the Example \ref{ex:alg2}, $(\Tp,\brac_+,\id_{T_+})$ is the tangent Lie algebroid of $\FF_+$, with the $\brac_+$ simply the restriction of the Lie bracket to $T_+$ and anchor the identity on $\Tp$. The additional differential structure given by the exterior derivative $\rd_+$, the Lie derivative $\Lie^+$ and the Dorfman bracket $\bbrac_+$, are given as follows.

Writing $\brac_+$ as $\brac$ to simplify the notation, we can express the action of $\rd_+$ and $\Lie^+$ on vectors in $\se(T_-)$ instead of one-forms in $\se(\Tp^*)$ using the isomorphism $\rho_+$ (see Definition \ref{eta_iso}). The Dorfman bracket of the Lie algebroid $T_+$, defined naturally on $T_+\oplus T_+^*$, can be consequently expressed on $T\PS=T_+\oplus T_-$ instead:
\begin{align*}
\bl x_++\xt,y_++\yt\br_+&=[x_+,y_+]+\eta^{-1}\left[\Lie^+_{x_+}\eta(\yt)-\Lie^+_{y_+}\eta(\xt)+\rd_+\eta(\xt,y_+)\right].
\end{align*}
This can be rewritten by contracting with $Z=\zp+\zt$:
\begin{align}
\begin{aligned}\label{Dorf_LL2}
\eta(\bl x_+\!\!+\!\xt,y_+\!\!+\!\yt\br_+,z_+\!+&\zt\!)\!=\!x_+[\eta(\yt,z_+\!)]\!-\!y_+[\eta(\xt,z_+\!)]\!+\!z_+[\eta(\xt,y_+\!)]\\
&\!+\!\eta([x_+,y_+],\zt)\!+\!\eta(\xt,[y_+,z_+])\!-\!\eta(\yt,[x_+,z_+]).
\end{aligned}
\end{align}
Therefore, \eqref{Dorf_LL2} defines a bracket on $T\PS$ satisfying
\begin{align*}
\rho_+ \bbl X,Y\bbr_+=\bbl \rho_+ X,\rho_+ Y\bbr_{\FF_+},
\end{align*}
where $\bracd_{\FF_+}$ is the standard Dorfman bracket on $(\TT)\FF_+$. In fact, $\rho_+$ induces a Courant algebroid structure on $T\PS$:
\begin{proposition}\label{prop:courant_alg_L}\sloppy{
Let $(\PS,\eta,K)$ be a $p$-para-Hermitian manifold.
Then $(T\PS,\PP,\eta,\bracd_+)$, where $\bracd_+$ is defined by \eqref{Dorf_LL2} and $\PP$ is the para-Hermitian projection \eqref{eq:projections}, is a Courant algebroid. Moreover, $\rho_+$ is an isomorphism of Courant algebroids
\begin{align*}
(T\PS,\PP,\eta,\bracd_+) \overset{\rho_+}{\longrightarrow}((\TT)\FF_+,\pi_+,\la\ ,\ \ra,\bbrac_{\FF_+}).
\end{align*}
This means that
\begin{align*}
\rho_+\bl X,Y\br_+=\bl \rho_+ X,\rho_+ Y\br_{\FF_+},\ \ \eta(X,Y)=\la \rho_+ X,\rho_+ Y\ra,\ \ \rho_+\PP(X)=\pi_+\rho_+(X).
\end{align*}
Analogous statement holds for $n$-para-Hermitian manifolds.}
\end{proposition}

\begin{proof}
We prove the statement for the $p$-para-Hermitian case, for the $n$-para-Hermitian case the proof is identical. We first show the morphism property of $\rho_+$.

$\rho_+\bl X,Y\br_+=\bl \rho_+ X,\rho_+ Y\br_{\FF_+}$ is satisfied by construction of $\bracd_+$,
\begin{align*}
\langle \rho_+(X),\rho_+(Y) \rangle=\langle x_++\eta(\xt),y_++\eta(\yt)\rangle=\eta(x_+,\yt)+\eta(\xt,y_+)=\eta(X,Y),
\end{align*}
and $\rho_+\PP(X)=\pi_+\rho_+(X)$ is trivial.

The fact that $\eta$, $\PP$ and $\bracd_+$ satisfy the axioms of a Courant algebroid is now a direct consequence of the fact that $\la\ ,\ \ra$, $\pi_+$ and $\bbrac_{\FF_+}$ do. Explicitly, let $e_i,\ i=1,\cdots,3$ be sections of $\Tp\oplus \Tp^*$ and $X_i,\ i=1,\cdots,3$ be sections of $T\PS$ such that $\rho(X_i)=e_i$. Using that $\bbrac_{\FF_+}$ satisfies $(1)$ in Definition \ref{def:courant}, we have
\begin{align*}
\pi_+(e_1)\langle e_2,e_3\ra=\PP(X_1)\eta(X_2,X_3)=&\la \bbl e_1,e_2\bbr_{\FF_+},e_3\ra+\la e_2,\bbl e_1,e_3\bbr_{\FF_+}\ra\\
=&\eta(\bl X_1,X_2\br_+,X_3)+\eta(X_2,\bl X_1,X_3\br_+),
\end{align*}
which shows that $\bracd_+$ also satisfies the property $(1)$. Similarly for the property $(2)$. The property $(3)$ follows from the same argument;
\begin{align*}
\bbl{e_1},\bbl{e_2},e_3\bbr_{\FF_+}\bbr_{\FF_+}=\bbl{e_1},\rho_+(\bl{X_2},X_3\br_+)\bbr_{\FF_+}=\rho(\bl{X_1},\bl{X_2},X_3\br_+\br_+),
\end{align*}
which is equal to $\rho(\bl\bl{X_1},X_2\br_+X_3\br_++\bl{X_2},\bl{X_1},X_3\br_+)$ by an analogous calculation.
\end{proof}

\begin{notation*}
In the following we will denote the Dorfman brackets of $(\TT)\FF_\pm$ by $\bracd_{\FF_\pm}$ and their counterparts on $T\PS$ by $\bracd_\pm$.
\end{notation*}
\section{Adapted Connections, Courant algebroids and the D-bracket}\label{sec:adapted}
In the Section \ref{couralgds} we have seen that as long as one of the distributions $T_\pm$ of an almost para-Hermitian manifold is integrable, the tangent bundle $T\PS$ acquires the structure of a Courant algebroid $(T\PS,\PPpm,\eta,\bracd_\pm)$. Because this Courant algebroid is isomorphic to the standard Courant algebroid over the foliation of an integrable distribution $((\TT)\FF_\pm,\pi_T,\la\ ,\ \ra,\bracd)$, we can use the result of Proposition \ref{ex:dorfman_connec} to express the bracket $\bracd_\pm$ using a torsionless connection on $\FF_\pm$. Moreover, such expression for the bracket is still well defined even when none of the distributions $T_\pm$ are integrable, but the bracket is no longer a Courant algebroid bracket; in particular, it fails to satisfy the Jacobi identity\footnote{Such bracket defines a metric algebroid \cite{vaisman2012geometry}}. This procedure can then be generalized to define the D-bracket for a non-integrable almost para-Hermitian manifold. As we will see in Section \ref{sec:fluxes}, this proves to be essential for introducing the twisted D-bracket, because the fluxes twisting the bracket arise precisely as a consequence of non-integrability.
%

We start by recalling the result of Proposition \ref{ex:dorfman_connec} and applying it to the setting of para-Hermitian manifolds, leading to the following definitions
\begin{Def}\label{def:associated_bracket}
Let $(\PS,\eta,K)$ be an almost para-Hermitian manifold and $\n$ a connection on $T\PS$. We define the \textbf{bracket associated to $\n$} by
\begin{align}\label{Dbrac_unprojected}
\eta(\bl X,Y\br^{\n},Z)\coloneqq \eta(\n_{X}Y-\n_{Y}X,Z)+\eta(\n_{Z}X,Y),
\end{align}
and the \textbf{$\PPpm$-projected brackets associated to $\n$} by
\begin{align}\label{dorfman_projections}
\eta(\bl X,Y\br_\pm^{\n},Z)\coloneqq \eta(\n_{\PPpm(X)}Y-\n_{\PPpm(Y)}X,Z)+\eta(\n_{\PPpm(Z)}X,Y).
\end{align}
\end{Def}
An immediate consequence of $\PP+\PPt=\id$ is then
\begin{align}\label{eq:bracketsum}
\bracd^\n=\bracd^\n_++\bracd^\n_-.
\end{align}
We also define
\begin{Def}\label{def:adapted_conn}
Let $(\PS,\eta,K)$ be an almost para-Hermitian manifold and let $\n$ be a connection such that
\begin{align}\label{eq:adapted-con-properties}
\begin{aligned}
&1)\ \n_\xp\eta=0\\
&2)\ \n_{x_+} y_- \subset T_-,\ \forall x_+ \in \se(T_+),y_- \in \se(T_-)\\
&3)\ \eta(T^\n(x_+,y_+),z_-)=0\\
&4)\ \eta(T^\n(x_+,y_+),z_+)+\eta(\n_{z_+}x_+,y_+)=0,
\end{aligned}
\end{align}
where $T^\n$ is the torsion tensor of $\n$. We call $\n$ a \textbf{p-adapted connection} of the almost para-Hermitian manifold $(\PS,\eta,K)$. Analogously, we call $\n$ an \textbf{n-adapted connection} if conditions $1)-4)$ are satisfied with the roles of $T_\pm$ exchanged. If $\n$ is both p- and n-adapted, we call $\n$ simply an \textbf{adapted connection}.
\end{Def}
\begin{Ex}
Even though the properties $1)-4)$ of Definition \ref{def:adapted_conn} are hard to understand geometrically, there exists a special case of adapted connections which are much simpler: any para-Hermitian connection $\n$ ($\n \eta=\n K=0$) whose torsion along the eigenbundles $T_\pm$ vanishes, $T^\n(x_\pm,y_\pm)=0$, is an adapted connection.
\end{Ex}
The following theorem justifies the Definitions \ref{def:associated_bracket} and \ref{def:adapted_conn}.

\begin{theorem}\label{prop:metric-courant-TP}
Let $(\PS,\eta,K)$ be a p(n)-para-Hermitian manifold and $\n$ a connection on $T\PS$. The $\PPpm$-projected bracket associated to $\n$ is equal to the Courant algebroid bracket $\bracd_+$,
\begin{align*}
\bracd_+^\n=\bracd_+,
\end{align*}
if and only if $\n$ is a p(n)-adapted connection.
\end{theorem}
\begin{proof}
We prove the statement for $\bracd^{\n}_+$, the proof for $\bracd^{\n}_-$ is analogous. For the forward implication, we use the properties \eqref{eq:adapted-con-properties}, the integrability of $T_+$ and the fact that both $T_+$ and $T_-$ are isotropic with respect to $\eta$ to directly compute the expression for $\bracd^{\n}_+$:
\begin{align*}
\eta(\bl X,Y\br^{\n}_+,Z)&=\eta(\n_\xp Y-\n_\yp X,Z)+\eta(\n_\zp X,Y)\\
&=\eta([x_+,y_+],\zt)+\eta(\n_{x_+}\yt-\n_{y_+}\xt,z_+)\\
&+\eta(\n_{z_+}x_+,\yt)+\eta(\n_{z_+}\xt,y_+)\\
&=\eta([x_+,y_+],\zt)+x_+[\eta(\yt,z_+)]-\eta(\yt,\n_{x_+}z_+)-y_+[\eta(\xt,z_+)]\\
&+\eta(\xt,\n_{y_+}z_+)+\eta(\n_{z_+}x_+,\yt)+z_+[\eta(\xt,y_+)]-\eta(\xt,\n_{z_+}y_+)\\
&=x_+[\eta(\yt,z_+)]-y_+[\eta(\xt,z_+)]+z_+[\eta(\xt,y_+)]\\
&+\eta([x_+,y_+],\zt)+\eta(\xt,[y_+,z_+])-\eta(\yt,[x_+,z_+]),
\end{align*}
which coincides with the expression \eqref{Dorf_LL2} for the Dorfman bracket $\bracd_+$. Therefore, if $\n$ is adapted, we get $\bracd^\n_+=\bracd_+$.

We now show the converse statement, i.e. that the formula \eqref{dorfman_projections} gives the brackets $\brac_+$ only if $\n$ is adapted. On one hand, we have by \eqref{Dorf_LL2}
\begin{align*}
\eta(\bl \xp,\yp\br_+,\zt)=\eta([\xp,\yp],\zp)=0,
\end{align*}
but on the other hand, by \eqref{dorfman_projections},
\begin{align*}
\eta(\bl \xp,\yp\br_+^\n,\zt)=\eta(\n_\xp\yp-\n_\yp\xp,\zt),
\end{align*}
which means $\eta(T^\n(\xp,\yp),\zt)=0$. Similarly, we find
\begin{align*}
\eta(\bl \xp,\yt\br_+,\zt)=0=\eta(\bl \xp,\yt\br^\n_+,\zt)=\eta(\n_\xp\yt,\zt),
\end{align*}
which implies $\n_\xp\yt \in \se(T_-)$. Considering another component of the bracket, we find
\begin{align*}
\eta(\bl \xp,\yt\br_+,\zp)&=\xp\eta(\yt,\zp)-\eta(\yt,[\xp,\zp])\\
&=\xp\eta(\yt,\zp)-\eta(\yt,\n_\xp\zp-\n_\zp\xp),
\end{align*}
since $\eta(T^\n(\xp,\yp),\zt)=0$. This is by assumption equal to
\begin{align*}
\eta(\bl \xp,\yt\br^\n_+,\zp)=\eta(\n_\xp\yt,\zp)+\eta(\n_\zp\xp,\yt),
\end{align*}
which means that
\begin{align*}
\xp\eta(\yt,\zp)-\eta(\yt,\n_\xp\zp)=\eta(\n_\xp\yt,\zp),
\end{align*}
i.e. $\n_\xp \eta=0$. Lastly, we consider
\begin{align*}
\eta(\bl \xp,\yt\br^\n_+,\zp)=0&=\eta(\n_\xp\yp-\n_\yp\xp,\zp)+\eta(\n_\zp\xp,\yp)\\
&=\eta(T^\n(\xp,\yp),\zp)+N_+(\xp,\yp,\zp)+\eta(\n_\zp\xp,\yp),
\end{align*}
and because $T_+$ is integrable, we find $\eta(T^\n(\xp,\yp),\zp)+\eta(\n_\zp\xp,\yp)=0$.
\end{proof}
 
Theorem  \ref{prop:metric-courant-TP} in particular tells us that the brackets $\bracd_\pm$ can be obtained using for \textit{any} adapted connection $\n$. in particular shows that when $T_+$ is integrable, the bracket $\bracd_+^\n$ is related to the standard Dorfman bracket $\bracd_{\FF_+}$ on $(\TT)\FF_+$ by
\begin{align*}
\bl X,Y\br^\n_+=\rho_+^{-1}\bbl \rho_+ X,\rho_+ Y\bbr_{\FF_+},
\end{align*}
which means that $(T\PS,\PPpm,\eta,\bracd_\pm^\n)$ is the Courant algebroid of Proposition \ref{prop:courant_alg_L}. Moreover, the bracket $\brac_\pm$ can be obtained via the formula \eqref{dorfman_projections} using \textit{any} adapted connection $\n$.

The D-bracket is now defined as the (unprojected) bracket associated to an adapted connection:
\begin{Def} \label{def:D-bracket}
Let $(\PS,\eta,K)$ be an almost para-Hermitian manifold and $\n$ an adapted connection. We define the \textbf{D-bracket} by
\begin{align}\label{D-brac-con}
\eta(\bl X,Y\br^D,Z)\coloneqq \eta(\n_{X}Y-\n_{Y}X,Z)+\eta(\n_{Z}X,Y).
\end{align}
\end{Def}

The following corollary justifies our definition:
\begin{corollary}
On a para-Hermitian manifold, the D-bracket can be expressed as
\begin{align}\label{eq:D-brac-rho}
\bl X,Y \br^D\coloneqq \rho^{-1}_+\bbl \rho_+ X,\rho_+ Y \bbr_{\FF_+}+\rho^{-1}_-\bbl \rho_- X,\rho_- Y \bbr_{\FF_-},
\end{align}
where $\FF_\pm$ are the integral foliations of the eigenbundles of $K$ and $\bracd_{\FF_\pm}$ are the corresponding Dorfman brackets. If the manifold is also flat, then $\bracd^D$ takes the form \eqref{eq:D-bracket}:
\begin{align*}
\bl X,Y \br^D=\left(X^I\p_IY^J-Y^I\p_IX^J+\eta_{IL}\eta^{KJ}Y^I\p_KX^L\right)\p_J.
\end{align*}
\end{corollary}
\begin{proof}
The first statement follows from Theorem \ref{prop:metric-courant-TP} and \eqref{eq:bracketsum}. The second statement follows from the fact that a flat pseudo-Riemannian manifold $(M^{2n},\eta)$ is locally isomorphic to $\RR^{2n}$ and we can choose local coordinates so that $\eta$ is constant. Using the relationship to the Dorfman brackets $\bracd_{\FF_\pm}$ along with the coordinate expressions for the Lie bracket, Lie derivative and exterior derivative then yields the result.
\end{proof}

Using slightly different language, the definition of a D-bracket on a para-Hermitian manifold by \eqref{eq:D-brac-rho} has been presented in \cite{vaisman2013towards}.

We close the discussion with the following remarks
\begin{Rem}
Analogously to the Courant bracket, which is a skew version of the Dorfman bracket, we can define a skew version of the D-bracket, the C-bracket:
\begin{align}\label{C-bracket}
\bl X,Y\br^C\coloneqq \frac{1}{2}(\bl X,Y\br^D-\bl Y,X\br^D).
\end{align}
\end{Rem}

\begin{Rem}
On a para-Hermitian manifold, the C-bracket \eqref{C-bracket} can be written as a bracket operation on $T_+\oplus T_-\simeq(\TT)\FF_+$:
\begin{align*}
\bbl x+\ap,y+\bt\bbr^C&=[x,y]+\Lie_x\bt-\Lie_y\ap+\frac{1}{2}\rd_+(\langle \ap,y\ra-\la\bt, x\ra)\\
&+[\ap,\bt]^*+\Lie^*_\ap y-\Lie^*_\bt x +\frac{1}{2}\rd_- (\langle x,\bt \ra-\la y,\ap\ra),
\end{align*}
where $[\ap,\bt]^*\coloneqq \eta[\eta^{-1}\ap,\eta^{-1}\bt]$ and $\Lie^*_\ap y\coloneqq \eta^{-1}\Lie_{\eta^{-1}\ap}\eta(y)$. This is formally reminiscent of the bracket defined for a Lie bialgebroid in \cite{Liu:1995lsa} and because both the Lie bialgebroid bracket and the C-bracket are defined on a doubled space, the two brackets are sometimes confused. We would like to emphasize here that the brackets are, however, very different; in particular the Lie algebroids $T_+$ and $T_-\simeq T_+^*$ do not form a Lie bialgebroid and the C-bracket is not a Courant algebroid bracket on $T_+\oplus T_+^*$, because it fails to satisfy the analog of the Jacobi identity for the skew bracket. This has also been noted in \cite{vaisman2012geometry}, where a similar construction of the C-bracket using a metric connection has been presented for affine para-K\"ahler manifolds. There, the structure $(T\PS,\eta,\id,\bracd^C)$ has been given the name \textbf{metric algebroid}, reflecting the fact that the C-bracket is not a Courant algebroid bracket.
\end{Rem}

\subsection{The Canonical Connection}
We now define the canonical connection, which provides a concrete example of an adapted connection. Since any adapted connection gives the same bracket, an explicit example of such connection is especially useful in calculations involving the D-bracket. The importance of the canonical connection for the presented construction has been described in \cite{freidel2017generalised}, where it is used for defining the $\PPpm$-projected brackets.

\begin{Def}\label{def:canonical_conn}
Let $(\PS,\eta,K)$ be an almost para-Hermitian manifold and $\lc$ the Levi-Civita connection of $\eta$. We define the \textbf{canonical connection} $\n^c$ by
\begin{align}\label{eq:canonical_conn_contorsion}
\eta(\n^c_XY,Z)=\eta(\lc_XY,Z)-\frac{1}{2}\lc_X\omega(Y,KZ).
\end{align}
\end{Def}
It can be easily checked that the canonical connection is a para-Hermitian connection (i.e. $\n^c \eta=\n^c \omega=0$). This implies that $\n^c$ preserves the eigenbundles of $K$:
\begin{align*}
\n^c_X y_\pm \in \se(T_\pm),
\end{align*}
for $y_+\in \se(T_\pm)$. This property becomes explicit when $\n^c$ is rewritten in the form
\begin{align}\label{eq:canonica_conn_PP}
\n^c_XY=\PP\lc_X \PP Y+\PPt\lc_X\PPt Y.
\end{align}

We can now use $\n^c$ to write the D-bracket explicitly in terms of $\lc$ and $\rd \omega$:
\begin{proposition}\label{prop:Dbrac-contorsion}
The D-bracket of an almost para-Hermitian manifold is given by
\begin{align*}
\eta(\bbl X,Y\bbr^D,Z)=\eta(\bbl X,Y\bbr^{\lc},Z)-\frac{1}{2}[\ &\rd\omega^{(+3,-0)}(X,Y,Z)+\rd\omega^{(+2,-1)}(X,Y,Z)\\
-&\rd\omega^{(+0,-3)}(X,Y,Z)-\rd\omega^{(+1,-2)}(X,Y,Z)\ ],
\end{align*}
where $\bracd^{\lc}$ is the bracket associated to $\lc$.
\end{proposition}
\begin{proof}
This formula is easily derived from the definition of the D-bracket and the expression \eqref{eq:canonical_conn_contorsion} for the canonical connection.
\end{proof}

On a para-K\"ahler manifold, the canonical connection coincides with the Levi-Civita connection of $\eta$:
\begin{lemma}
Let $(\PS,\eta,K)$ be a para-K\"ahler manifold. Then the canonical connection coincides with the Levi-Civita connection of $\eta$, $\n^c=\lc$, hence $\bracd^D=\bracd^{\lc}$.
\end{lemma}
\begin{proof}
This is a direct consequence of \eqref{eq:canonical_conn_contorsion}, Lemma \ref{lem:herm-kahl} and Proposition \ref{prop:Dbrac-contorsion}.
\end{proof}

\begin{Rem}
The canonical connection appears in \cite{Ivanov}, where the authors introduce a class of para-Hermitian connections $\n^t$ parametrized by $t\in \RR$. This class also includes the Chern connection and the Bismut connection of a para-Hermitian manifold. The canonical connection is given by $\n^{t=0}$ and all connection in this class degenerate to the canonical connection on a nearly para-K\"ahler manifold.
\end{Rem}

\subsection{A Comment About The Section Condition}
In the physics literature, it is observed that even though the D-bracket is not a Courant algebroid bracket, this can be locally fixed by imposing the \textbf{section condition}, i.e. restricting the dependence of the vector field to half-dimensional submanifolds called polarisations.
The most commonly discussed polarisation is the one given by $\tl{x}_i=const.$ for all $i=1,\cdots,n$, or equivalently by setting
\begin{align*}
\pt^i=0,\quad\forall\ i=1,\cdots,n.
\end{align*}
In our framework of para-Hermitian manifolds, instead of understanding the section condition locally as an independence on certain subset of local coordinates, equations \eqref{Dbrac_unprojected} and \eqref{dorfman_projections} provide an approach to impose the section condition globally.


The following result shows that on flat manifolds our approach coincides with globally imposing the section condition in the usual manner by restricting the dependence of the vector fields to half of the adapted coordinates, i.e. by only considering vector fields parallel along one of the foliations:
\begin{proposition}
Let $(\PS,\eta,K)$ be a flat para-Hermitian manifold and let $X_+$ and $Y_+$ be vector fields that are parallel along $\FF_-$. Then
\begin{align*}
\bl X_+,Y_+\br^D=\bl X_+,Y_+\br_+,
\end{align*}
or equivalently
\begin{align*}
\bl X_+,Y_+\br_-=0.
\end{align*}
\end{proposition}
\begin{proof}
Let $(x_i,\tl{x}^j)$ be the adapted local coordinates, where $x^i$ are coordinates along $\FF_+$. Then $X_+$ is locally expressed as
\begin{align*}
X_+=X^i(x)\p_i+\tl{X}_j(x)\pt^j,
\end{align*}
and similarly for $Y_+$. We will now use the canonical connection in the definition of the D-bracket. Because $\n^c_{\PPt X_+} Y^j(x)=\PPt(X_+)[Y^j(x)]=\tl{X}_i\pt^i[Y^j(x)]=0$ (similarly for $\tl{Y}_j(x)$), we have
\begin{align*}
\n^c_{\PPt X}Y_+=Y^i(x)(\n^c_{\PPt Y}\p_i)+\tl{Y}_j(\n^c_{\PPt Y}\pt^j).
\end{align*}
Using the definition of $\n^c$, we find
\begin{align*}
\n^c_{\PPt X}\p_i=\lc_{\PPt X}\p_i=0,
\end{align*}
because $\lc$ is flat and the Christoffel symbols of $\lc$ vanish. Similarly, $\n^c_{\PPt X}\pt^j=0$, which shows that $\n^c_{\PPt X}Y_+=0$. Recalling \eqref{dorfman_projections} yields $\bracd_-=0$.
\end{proof}
We comment that this global approach could be a step towards relaxing the section condition partially or altogether, which is deemed desirable in various discussions about the foundations of the DFT (see for example \cite{geissbuhler2011double,aldazabal2011effective,berman2014supersymmetry}, for a different approach \cite{lee2016towards}).

\section{Fluxes and Deformations of Para-K\"ahler Manifolds}\label{sec:fluxes}
In this section we show how to acquire the twisted D-bracket using the results of the previous sections. More concretely, we define a certain class of deformations of the para-Hermitian structure $K$ and show that in the case when $K$ is para-K\"ahler, the D-bracket corresponding to the deformed structure is the D-bracket twisted by fluxes described in the DFT literature. The fluxes appear here as obstructions to integrability with respect to the D-bracket of $K$ (see Definition \ref{def:weak-int}).

This shows that the language of fluxes used in DFT (and in general in string theory) can be included in the framework of para-Hermitian manifolds and that the twisted bracket arises as a consequence of a deformation of the underlying geometry. For works discussing the fluxes and twisted brackets from a different point of view, see \cite{Shelton:2005cf,andriot2012non,blumenhagen2012bianchi,halmagyi2009non,Aldazabal:2013sca,geissbuhler2013exploring} and references therein. A mathematical overview of related concepts is given in \cite{kosmann:quasi}.

\subsection{B-transformation of a Para-Hermitian Structure}
We first define the notion of a B-transformation for any almost para-Hermitian manifold:
\begin{Def}
Let $(\PS,\eta,K)$ be an almost para-Hermitian manifold. \\ A \textbf{B-transformation} is an endomorphism of $T\PS$, given in the splitting $T_+\oplus T_-$ by
\begin{align*}
e^{B_+}\coloneqq
\begin{pmatrix}
\id & 0 \\
B_+ & \id
\end{pmatrix} \in \End(T\PS)
\end{align*}
where $B_+:T_+\rightarrow T_-$ is a skew map such that $\eta(B_+X,Y)=-\eta(X,B_+Y)$. The induced map on the endomorphisms of $T\PS$, given by conjugation, is also called a \textbf{${B_+}$-transformation} and in particular the ${B_+}$-transformation of $K$ is given by
\begin{align*}
K\overset{e^{B_+}}{\longmapsto} K_{B_+}=e^{B_+} K e^{-{B_+}}.
\end{align*}
We also say that $K_{{B_+}}$ is the \textbf{${B_+}$-transformation} of $K$. 
\end{Def}
It is easy to see that ${B_+}$ can be given by either a two-form $b_+$ or a bivector $b_-$:
\begin{align}\label{eq:b-beta}
\eta(BX,Y)=b_+(X,Y)=b_-(\eta (X),\eta (Y)),
\end{align}
where $b_+$ is of type $(+2,-0)$ and $b_-$ is of type $(+0,-2)$, so we can write $b_+(X,Y)=b_+(\xp,\yp)$.

Similarly, we can define a map $B_-:T_-\rightarrow T_+$ given by a type $(+0,-2)$ two-form $\bt_-$ or a $(+2,-0)$ bivector $\bt_+$. The $B_-$-transformation then takes the form
\begin{align}\label{rem:Bminus}
e^{B_-}\coloneqq
\begin{pmatrix}
\id & B_- \\
0 & \id
\end{pmatrix} \in \End(T\PS).
\end{align}
Without loss of generality, we will continue the discussion for $B_+$. The case when both such transformations are performed simultaneously is left for future work.

\begin{Rem}
\sloppy{
As we will see later, the B-transformation by $B_+$ corresponds to a $b$-field transformation of $(\TT)\FF_+$ and a $\beta$-field transformation of $(\TT)\FF_-$, while the transformation by $B_-$ corresponds to a $\beta$-field transformation of $(\TT)\FF_+$ and a $b$-field transformation on $(\TT)\FF_-$.}

In physics, the manifold $\FF_+$ represents the usual space-time, while $\FF_-$ represents a \textit{dual space-time}. We therefore denote the corresponding two-forms and bivectors from the point of view of the space-time $\FF_+$: $b_\pm$ for $B_+$ and $\bt_\pm$ for $B_-$.
\end{Rem}

\begin{notation*}
From now on we will only discuss the B-transformation given by $B_+$, denoted simply as $B_+=B$. Moreover, since both $b_\pm$ give the same data of the map $B$ and are in the adapted coordinate frame given by the same coefficient functions,
\begin{align*}
b_+=b_{ij}dx^i\w dx^j,\quad b_-=b_{ij}\pt^i\w \pt^j
\end{align*}
we will denote $b_\pm\coloneqq b$ where no confusion is possible.
\end{notation*}

In the splitting $T_+\oplus T_-$, $K_B$ is given by
\begin{align*}
K_{B}=
\begin{pmatrix}
\id & 0 \\
2B & -\id
\end{pmatrix},
\end{align*} 
where we observe that $K_{B}^2=\id$. Moreover, because of the skewness property of $B$, $K_{B}$ also satisfies $\eta(K_{B}\cdot,K_{B}\cdot)=-\eta$ and therefore the B-transformation of $K$ is a new almost para-Hermitian structure  on $\PS$. The action of $e^B$ on $K$ can also be seen as a change of the fundamental form:
\begin{align}\label{omegaB}
\omega\overset{e^B}{\longmapsto} \omega_B=\eta K_B=\omega+2b.
\end{align}

The bigrading of tensors \eqref{eq_plusminus_decomp} with respect to $K_B$ will be denoted by $(+p,-k)_B$. The corresponding projections act on vectors by $\PPBpm\coloneqq \frac{1}{2}(\id \pm K_B)$:
\begin{align*}
\PPB(X)=x_++B(x_+),\quad \PPtB(X)=\xt-B(x_+).
\end{align*}
Because $B$ maps $T_+$ to $T_-$, $\im(\PPtB)=T_-$, but $\im(\PPB)=T_+^B\neq T_+$, i.e $K$ and $K_B$ share the $-1$ eigenbundle, but the $+1$ eigenbundles are different. This means that even if $K$ is integrable, $K_B$ need not be. We therefore view $K_B$ as a deformation of $K$. It turns out that the D-bracket gives rise to a Maurer-Cartan type equation relevant to this deformation problem.

\subsection{Weak Integrability and the Maurer-Cartan Equation}
We now introduce the notion of weak integrability for an isotropic distribution on an almost para-Hermitian manifold, where in contrast to the usual notion of Frobenius integrability, the Lie bracket is replaced by the D-bracket.

\begin{Def}\label{def:weak-int}
Let $(\PS,\eta,K)$ be an almost para-Hermitian manifold and $\bracd^D$ the associated D-bracket. We say an $\eta$-isotropic distribution $\mathcal{D}$ is \textbf{weakly integrable} (with respect to $K$) if it is involutive under the D-bracket of $K$, i.e.
\begin{align*}
\bl\mathcal{D},\mathcal{D}\br^D\subset \mathcal{D}.
\end{align*}
\end{Def}
The eigenbundles $T_\pm$ of $K$ are always weakly integrable with respect to $K$:
\begin{corollary}
Let $(\PS,\eta,K)$ be an almost para-Hermitian manifold. Then the eigenbundles $T_\pm$ of $K$ are weakly integrable with respect to $K$ itself.
\end{corollary}
\begin{proof}
This is an consequence of the definition of D-bracket \eqref{D-brac-con}, along with the fact that the canonical connection $\n^c$, preserves the eigenbundles $T_\pm$, which are isotropic with respect to $\eta$.
\end{proof}

However, the eigenbundles of a B-transformed para-Hermitian structure are in general not weakly integrable with respect to $K$. Because the weak integrability is a relative notion depending on a choice of a \textit{reference} almost para-Hermitian structure which defines the D-bracket, we introduce the following
\begin{Def}\label{def:compatible}
Let $(K,\eta)$ and $(K',\eta)$ be two almost para-Hermitian structures on a manifold $\PS$. We say \textbf{$K'$ is compatible with $K$} if the eigenbundles of $K'$ are weakly integrable with respect to $K$.
\end{Def}
Therefore, any almost para-Hermitian structure is always compatible with itself.
We will now study the compatibility of $K_B$ with $K$. Since $K_B$ is completely determined by $(K,\eta)$ and $b$, its compatibility with $K$ can be expressed purely in terms of $b$.

We first prove the following formula:
\begin{lemma}\label{lem:formula-KB}
Let $K_B$ be a B-transformation of a para-Hermitian structure $(\eta, K)$ and let $\PPB=\frac{1}{2}(\id+K_B)$ be the projection onto the $+1$-eigenbundle of $K_B$. Then
\begin{align*}
\eta(\bl \PPB X,\PPB Y\br^D,\PPB Z)=\rd_+b+(\Lambda^3\eta)[b,b]_-,
\end{align*}
where $\rd_+$ is the Lie algebroid differential of $T_+$, $\brac_-$ is the Schouten bracket of $T_-$ and $b$ denotes both the two-form and bivector \eqref{eq:b-beta} corresponding to B.
\end{lemma}
\begin{proof}
First, we recall that $\n^c=\PP \lc\PP+\PPt \lc \PPt$ and note that
\begin{align*}
\eta(\n^c_{\PPB X}\PPB Y,\PPB Z)&=\eta(\lc_{\PPB X}\yp,B(\zp))+\eta(\lc_{\PPB X}B(\yp),\zp)\\
&=b(\zp,\lc_{\PPB X}\yp)+\lc_{\PPB X}b(\yp,\zp)+b(\lc_{\PPB X}\yp,\zp)\\
&=\lc_{\xp+B(\xp)}b(\yp,\zp),
\end{align*}
where we used $\eta(BX,Y)=b(X,Y)$. We now use $\n^c$ for the D-bracket along with the above calculation,
\begin{align*}
\eta(\bl \PPB X,\PPB Y\br^D,\PPB Z)=\mkern-18mu\sum_{Cycl.\ \xp,\yp,\zp}\mkern-18mu\lc_{\xp}b(\yp,\zp)+\lc_{B(\xp)}b(\yp,\zp),
\end{align*}
and applying Lemma \ref{lem:schouten} along with the relationship \eqref{eq:b-beta} yields the result.
\end{proof}
An immediate consequence of the above is
\begin{proposition}\label{prop:weak-int-maurer-cartan}
Let $(K_B,\eta)$ be a B-transformation of a para-Hermitian structure $(K,\eta)$. Then $K_B$ is compatible with $K$ if and only if
\begin{align}\label{eq:maurer-cartan}
\rd_+b+(\Lambda^3\eta)[b,b]_-=0.
\end{align}
\end{proposition}
\begin{proof}
$T_+^B$ is weakly integrable with respect to $K$ if and only if
\begin{align*}
\eta(\bl \PPB X,\PPB Y\br^D,\PPB Z)=0,
\end{align*}
since $T_+^B$ is maximally isotropic with respect to $\eta$ and $K_B$ and $K$ share the $-1$-eigenbundle. The result of lemma \ref{lem:formula-KB} then yields the equation \eqref{eq:maurer-cartan}.
\end{proof}



In light of Proposition \ref{prop:weak-int-maurer-cartan}, we argue that the notion of weak integrability is a more natural condition to consider in the context of deformations of para-Hermitian manifolds than Frobenius integrability since it is equivalent to the Maurer-Cartan type equation \eqref{eq:maurer-cartan}. Exploring the geometrical meaning of weak integrability in more detail as well as its relationship to Frobenius integrability is an interesting problem for future work.

\begin{Rem}
The presented results share many similarities with the deformation theory of Dirac structures \cite{gualtieri2017deformation}. This is not entirely surprising, since the distribution $T_+^B$ can be seen as an almost Dirac structure for either of the two Courant algebroids $(T\PS,\eta,\PPpm,\bracd_\pm)$.
\end{Rem}

\subsection{Para-K\"ahler Manifolds and Fluxes}
We now restrict our discussion to the case when $(K,\eta)$ is para-K\"ahler, meaning that the fundamental form $\omega=\eta K$ is closed. In this case the compatibility of $K_B$ with $K$ is easily related to the Frobenius integrability. Moreover, the D-bracket of $(K_B,\eta)$ is the twisted D-bracket and the fluxes are given by the obstruction to the compatibility of $K_B$ with $K$.

\begin{lemma}\label{lem:3-0B}
Let $(K_B,\eta)$ be a B-transformation of a para-K\"ahler structure. Then
\begin{align*}
\rd_+b+(\Lambda^3\eta)[b,b]_-=\rd b^{(+3,-0)_B},
\end{align*}
where $(+3,-0)_B$ denotes the bigrading \eqref{eq_plusminus_decomp} with respect to $K_B$.
\end{lemma}
Before proving Lemma \ref{lem:3-0B}, we first prove the following
\begin{lemma} \label{lem:lc_b}
Let $(K_B,\eta)$ be a B-transformation of a para-Hermitian structure $(\PS,\eta,K)$ and denote $B=\eta b$. Then $\n_X b$, where $\n$ is an adapted connection, is a type $(2,0)$ form with respect to $K$ for any vector field $X$.
\end{lemma}
\begin{proof}
Since $b$ itself is type $(+2,-0)$ and $\n$ preserves $T_\pm$,
\begin{align*}
\n_Xb(Y,Z)=&Xb(Y,Z)-b(\n_XY,Z)-b(Y,\n_XZ)\\
=&Xb(y_+,z_+)-b(\n_Xy_+,z_+)-b(y_+,\n_Xz_+)=\n_Xb(y_+,z_+).
\end{align*}
\end{proof}
\begin{proof}[Proof of Lemma \ref{lem:3-0B}]
$\n^c$ is an adapted connection and for a para-K\"ahler manifold, $\n^c=\lc$. Therefore, by Lemma \ref{lem:lc_b}, $\lc_X b$ is a two-form of type $(+2,-0)$ and we compute
\begin{align*}
\rd b^{(+3,-0)_B}(X,Y,Z)&=\rd b(\PPB X,\PPB Y,\PPB Z)=\mkern-18mu\sum_{Cycl.\ X,Y,Z}\mkern-18mu\lc_{\PPB X}b(\PPB Y,\PPB Z)\\
&=\mkern-18mu\sum_{Cycl.\ \xp,\yp,\zp}\mkern-18mu\lc_{\xp+B(\xp)}b(\yp,\zp)
\end{align*}
which yields the result.
\end{proof}

Therefore, when $(K_B,\eta)$ is a B-transformation of a para-K\"ahler structure, by Lemma \ref{lem:30integr}, we see that integrability of $T_+^B$ implies weak integrability of $T_+^B$ with respect to $K$ and therefore compatibility of $K_B$ with $K$. If moreover $(K_B,\eta)$ is nearly para-K\"ahler, the Frobenius integrability is equivalent to the compatibility of $K_B$ with $K$.

The D-bracket for $(K_B,\eta)$ is given by the following:
\begin{proposition}\label{prop:twistedDbrac}
Let $K_B$ be a B-transformation of a para-K\"ahler structure $(\PS,\eta,K)$. Then the D-bracket associated to $K_B$ is given by
\begin{align}\label{twistedDbrac}
\eta(\bl X,Y\br^{D,B},Z)=\eta(\bl X,Y\br^D ,Z)-(\rd b)(X,Y,Z).
\end{align}
where $\bl\ ,\ \br^D$ denotes the D-bracket of $K$.
\end{proposition}
\begin{proof}
Because $\omega_B=\omega+2b$, and $K$ is K\"ahler ($\rd\omega=0$), we conclude that $\rd \omega_B=2\ \rd b$ and as a consequence of Lemma \ref{lem:lc_b} we find by a direct calculation that the only non-zero components of $\rd b$ are the $(+3,-0)_B$ and $(+2,-1)_B$ components. Proposition \ref{prop:Dbrac-contorsion} then yields the result.
\end{proof}

%

We now describe the different components of the twist $\rd b$ in terms of fluxes. For this, we need to write the components of $\rd b$ in the splitting corresponding to $K_B$; while the frame of $T\PS$ diagonalizing $K_B$ is $\{\p_i^B=\p_i+b_{ij}\pt^j,\pt_j\}$, the dual  frame of $T^*\PS$ is $\{dx^i,d\tl{x}^B_i=d\tl{x}_i+b_{ij}dx^j\}$:
\begin{align*}
\rd b&=\p_ib_{jk}dx^i\w dx^j\w dx^k+\pt^ib_{jk}d\tl{x}_i\w dx^j\w dx^k\\
&=\p_ib_{jk}dx^i\w dx^j\w dx^k+\pt^ib_{jk}d\tl{x}^B_i\w dx^j\w dx^k+b_{il}\pt^lb_{lk}dx^i\w dx^j\w dx^k.
\end{align*}

\begin{itemize}
\item The $(+3,-0)_B$ component of $\rd b$ is by Lemma \ref{lem:3-0B} given by
\begin{align*}
\rd b^{(+3,-0)_B}=\rd_+b+(\Lambda^3\eta)[b,b]_-
\end{align*}
where $\rd_+ b$ is also the $(+3,-0)$ component of $\rd b$ with respect to $K$. The two terms correspond to the well known \textbf{H-flux} and a (dual) \textbf{R-flux}
\begin{align*}
\rd_+b&=H=\p_ib_{jk}dx^i\w dx^j\w dx^k,\\
(\Lambda^3\eta)[b,b]_-&=\tl{R}=b_{il}\pt^lb_{lk}dx^i\w dx^j\w dx^k,
\end{align*}
The H-flux is a 3-form on $\FF_+$, while $[b,b]_-$ is a three-vector on $\FF_-$. In physics the R-flux is usually a three-vector on the space-time manifold (in our case $\FF_+$), which is why we call $(\Lambda^3\eta)[b,b]_-$ the dual R-flux. The usual R-flux on $\FF_+$ would then be a result of the $B_-$-transformation \eqref{rem:Bminus} corresponding to a bivector on $T_+$.
\item The $(+2,-1)_B$ component of $\rd b$ reads
\begin{align*}
\rd b^{(+2,-1)_B}=\pt^ib_{jk}d\tl{x}^B_i\w dx^j\w dx^k,
\end{align*}
In this expression we recognize the (dual) \textbf{Q-flux}. We again see that this expression has the opposite index structure to the usual Q-flux due to the fact that $b$ is a bivector on $\FF_-$ as opposed to $\FF_+$, hence the name dual Q-flux.
\end{itemize}
The remaining components of $\rd b$ vanish. We notice that the $H$ and $\tl{R}$ fluxes are related, giving the $(3,0)_B$ part of $\rd b$ (and therefore $\rd\omega_B$, by \eqref{omegaB} and $\rd \omega=0$), and therefore the obstruction to compatibility of $K_B$ with $K$. This obstruction,
\begin{align*}
\mathcal{H}_{ijk}=\p_{[i}b_{jk]}+b_{[il}\pt^{\small l}b_{jk]},
\end{align*}
is in the physics literature sometimes called the \textit{covariantized H-flux} or an \textit{H-flux without section condition} \cite{andriot2012non}. 
\begin{Rem}
We have seen that the all the fluxes correspond to the same data -- the map $B$, which can be seen either as a two-form or as a bivector -- and the resulting fluxes, given by $\rd_+ b$, $\lc_{x_-}b$ and $[b,b]_-$ are just different differential operations on $b$. This relationship between fluxes reflects what happens in physics, where the $H$, $Q$ and $R$ fluxes are related by \textit{T-duality}, which in simplified terms amounts to exchanging the individual $x_i$ and $\tl{x}^i$ coordinates, i.e. the coordinates between the $\FF_+$ and $\FF_-$ manifolds. For example, if one starts with the $H_{123}$ component of $H$, i.e. the component of $H$ along $x^1$, $x^2$ and $x^3$, after performing T-duality along each of these coordinates, one ends up with an R-flux along the T-dual coordinates $\tl{x}_1$, $\tl{x}_2$ and $\tl{x}_3$, $R_{123}$. In our case, this relationship is realized by the isomorphism of $\eta$ (and relabelling of coordinates).
\end{Rem}

Another standard flux appearing in physics literature is the \textbf{f-flux}. This can be easily included in our discussion by the diagonal action on the tangent bundle frame given by
\begin{align}\label{structure-gp}
E_A=
\begin{pmatrix}
A & 0 \\
0 & (A^{-1})^*
\end{pmatrix} \in \End(T\PS).
\end{align}
Elements of this type form the structure group of para-Hermitian manifolds and therefore preserve $K$. Denote $e_a=[A]_a^i\p_i$ and $e^a=[(A^{-1})^*]_j^a\pt^j$. The f-flux then appears as
\begin{align*}
\eta(\bl e_a,e_b \br^D,e^c)=f_{ab}^c.
\end{align*}

%

\subsection{Interpretation as a Generalized Geometry on $(\TT) \FF_\pm$}\label{sec:GG-relationship}
The B-transformation can be related to the $b$-field transformation and the $\beta$-field transformation of the generalized tangent bundles $(\TT)\FF_\pm$ and the corresponding Dorfman brackets as follows:

On a p-para-Hermitian manifold, $e^B$ gives a $b$-field transformation of $(\TT)\FF_+$ by $b_+$ via $e^{b_+}\rho_+=\rho_+e^B$:
\begin{align*}
e^{b_+}=\begin{pmatrix}
\id & 0 \\
b_+ & \id
\end{pmatrix} \in \End((\TT)\FF_+).
\end{align*}
This corresponds to a change of a splitting of the exact sequence of vector bundles over $\FF_+$
\begin{align*}
0\longrightarrow T^* \overset{\tl{i}}{\longrightarrow}\TT \overset{p}{\longrightarrow}T\longrightarrow 0,
\end{align*}
on the right; we modify the natural splitting given by the inclusion $i:T\hookrightarrow \TT$ by a two-form to get a new splitting:
\begin{align*}
 i_{b_+}=i+\tl{i}\circ {b_+}: T\rightarrow \TT,\quad \xp\mapsto \xp+b_+(\xp).
\end{align*}
The corresponding Dorfman bracket is then acquired from \eqref{twistedDbrac} by projecting the derivatives:
\begin{align*}
\eta(\bl X,Y\br^B_+,Z)&=\eta(\n^{c,B}_{\PP X}Y-\n^{c,B}_{\PP Y}X,Z)+\eta(\n^{c,B}_{\PP Z}X,Y)\\
&=\eta(\bl X,Y\br_+,Z)-\rd_+b_+(\xp,\yp,\zp),
\end{align*}
where $\bl X,Y\br_+$ is the usual Dorfman bracket \eqref{dorfman_projections}. Mapping by $\rho_+$ then recovers the \textbf{H-twisted} Dorfman bracket on $(\TT)\FF_+$.

Similarly, on an n-para-Hermitian manifold, $e^{B}$ gives a $\beta$-field transformation of $(\TT)\FF_-$ by $b_-$ via $e^{b_-}\rho_-=\rho_-e^{B}$:
\begin{align*}
e^{b_-}=\begin{pmatrix}
\id & b_- \\
0 & \id
\end{pmatrix} \in \End((\TT)\FF_-).
\end{align*}
\section{Example: Tangent Bundle of a Riemannian Manifold}
Here we present simple well-known example \cite{Szabo-paraherm,tangentbundle} of a para-Hermitian structure and apply the introduced formalism explicitly. This example as well as other examples for the cotangent bundle and the Drinfel'd doubles are in this context more broadly discussed in \cite{Szabo-paraherm}.

Let $(M,g)$ be an $n$-dimensional Riemannian manifold. The total space of the tangent bundle, $TM$ inherits an n-para-K\"ahler structure in the following way. The $T_-$ distribution is given by the vertical distribution of the tangent bundle projection $\pi: TM\rightarrow M$, i.e. at any point of $TM$ it is spanned by the vectors tangent to the fibres of $\pi$. The distribution $T_+$ is given by a choice of a horizontal distribution; for this we choose a metric connection $\n$, $\n g=0$ and $T_+$ is defined as the horizontal subbundle of $T(TM)$ with respect to $\n$. In local coordinates $(x^i,v^i)$, the splitting $T(TM)=T_+\oplus T_-$ is given by
\begin{align}\label{ex:TM1}
T_+=span\left\{\p_i-\Gamma_{ij}^k v^j \frac{\p}{\p v^k}\coloneqq H_i\right\}_{i=1\cdots n},\quad T_-=span \left\{\frac{\p}{\p v^i}\coloneqq V_i\right\}_{i=1\cdots n},
\end{align}
where $\p_i=\frac{\p}{\p x^i}$ and the Christoffel symbols $\Gamma_{ij}^k$ are defined as $\n_{\p_i}\p_j=\Gamma_{ij}^k\p_k$. To define the split signature metric $\eta$, we simply put
\begin{align*}
\eta(H_i,H_j)=\eta(V_i,V_j)=0, \quad \eta(V_i,H_j)=\eta(H_i,V_j)=g_{ij}.
\end{align*}
We find that the frame $\{H^i,V^i\}_{i=1\cdots n}$ of $T_\pm^*$ dual to $\{H_i,V_i\}_{i=1\cdots n}$ is given by
\begin{align*}
H^i=dx^i,\quad  V^i=dv^i+\Gamma_{kj}^i v^jdx^k,
\end{align*}
so that $H^i(H_j)=V^i(V_j)=\delta^i_j$ and $H^i(V_j)=V^i(H_j)=0$. We can use this frame to write $\eta, \omega$ and $K$ in explicitly as
\begin{align*}
\eta=g_{ij}(V^i\otimes H^j+H^i\otimes V^j),\ \omega=g_{ij}H^i\w V^j,\ K=H_i\otimes H^i-V_i\otimes V^i.
\end{align*}
\paragraph*{Integrability}
The vertical distribution $T_-$ is clearly integrable with the integral leaves being the fibres $\pi^{-1}(x)$, $x\in M$. The horizontal distribution $T_+$, on the other hand, is not integrable, as
\begin{align*}
[H_i,H_j]=R^k_{ijl} v^l V_k,
\end{align*}
where $R^k_{ijl}$ are components of the Riemann curvature tensor of $\n$. The obstruction to integrability of $T_+$ is therefore given by the curvature of $\n$. We can also infer this from the expression for the exterior derivative of $\omega$
\begin{align*}
\rd \omega=T_{ij}^k g_{kl} H^i\w H^j\w V^l+\frac{1}{2}R_{ijkl}v^l H^i\w H^j \w H^k,
\end{align*}
which also shows that $\rd \omega$ has no $(+1,-2)$ or $(+0,-3)$ components, therefore the above defined structure is n-para-K\"ahler. Moreover, it is para-K\"ahler if and only if $\n$ is torsionless and its curvature vanishes. In other words, the above construction defines a para-K\"ahler structure on $TM$ if and only if $\n$ is the Levi-Civita connection of $g$ and $g$ is a flat metric.

\paragraph*{B-transformation and the D-bracket}
We will now assume that $g$ is flat and therefore $\n^g$, the Levi-Civita connection of $g$, induces a para-K\"ahler structure on the total space of $TM$. We can choose a coordinate system on $M$ in which the metric $g$ has constant coefficients and the Christoffel symbols vanish identically, implying $H_i=\p_i$ and $V^i=dv^i$. This also means that the Christoffel symbols of the Levi-Civita connection $\lc$ of $\eta$ vanish. The D-bracket $\bracd^D$ then takes the form \eqref{eq:D-bracket}, where the capital indices denote the splitting into horizontal and vertical components, 
\begin{align*}
X^I\p_I=X^i\p_i+\tl{X}^j\frac{\p}{\p v^j}.
\end{align*}
To twist the bracket, we choose a two-form $b=b_{ij}(x,v)dx^i\w dx^j$, such that $b \in \Omega^{(+2,-0)}(TM)$ and B-transform the para-K\"ahler structure on $TM$. The new (almost-)para-Hermitian structure is given by
\begin{align*}
\omega_B&=g_{ij}dx^i\w dv^j+2b_{ij}dx^i\w dx^j, \\
K_B&=\p_i\otimes dx^i-\frac{\p}{\p v^i}\otimes dv^i+2b_{ik}g^{kj}\frac{\p}{\p v^j}\otimes dx^i,
\end{align*}
the $+1$ eigenbundle is now given by
\begin{align}\label{ex:TM2}
T_+^B=span\left\{\p_i+b_{ik}g^{kj}\frac{\p}{\p v^j}\coloneqq H_i'\right\}_{i=1\cdots n},
\end{align}
which means that $T_-^*$ is spanned by $V'^i=dv^i-b_{jk}g^{ki}dx^j$. The twisted D-bracket can be explicitly written as
\begin{align*}
\bl X,Y\br^{D,B}=\bl X,Y\br^D&-\p_ib_{jk}dx^i\w dx^j\w dx^k-\frac{\p}{\p v^i}b_{jk}dv^i\w dx^j\w dx^k\\
=\bl X,Y\br^D&-\p_ib_{jk}dx^i\w dx^j\w dx^k-\frac{\p}{\p v^i}b_{jk}V'^i\w dx^j\w dx^k\\
&-b_{im}g^{ml}\frac{\p}{\p v^l}b_{jk}dx^i\w dx^j\w dx^k,
\end{align*}
and we can read off the covariantized H-flux and Q-flux:
\begin{align*}
\mathcal{H}_{ijk}=\p_{[i}b_{jk]}+b_{[im}g^{ml}\frac{\p}{\p v^l}b_{jk]},\quad Q_{ijk}=\frac{\p}{\p v^i}b_{jk}.
\end{align*}

\section{Comments, Conclusions and Outlook}
We have shown that the D-bracket appearing in physics literature can be defined on a general almost para-Hermitian manifold. Moreover, we related it to a certain deformation problem of para-Hermitian structures and defined a twisted D-bracket as the D-bracket corresponding to the deformed structure. Even though the compatibility of the deformed structure with the original one is equivalent to a Maurer-Cartan type equation, a more thorough and complete analysis of the deformation theory of para-K\"ahler and para-Hermitian manifolds is needed to fully understand the problem, which is one of the goals of our future work. A closely related problem is to use the presented formalism of deformations and fluxes to include all the complicated fluxes found in the physics literature and interpret their various relationships. We also wish to tie our results to the very similar results found in the work on deformations of Dirac structures and pinpoint the exact relationship between the two geometric set-ups.

Another future research direction is once again motivated by physics. One crucial aspect of the DFT we have not discussed here are the T-duality and generalized coordinate transformations, leading to a notion of so called T-folds \cite{hull2005geometry} -- objects patched together from the local coordinate patches by $O(n,n)$ transformations that go beyond the elements of the para-Hermitian structure group \eqref{structure-gp}: One considers off-diagonal elements that include exactly the B-transformations we discussed, but also T-duality transformations that exchange the frame vectors $\p_i\leftrightarrow \pt^i$. One might be able to realize this type of off-diagonal transformations as non-integrable para-Hermitian structures where the local coordinates $(x^i,\tl{x}_j)$ cannot be found on each patch of the manifold such that the $x^i$ and $\tl{x}_j$ are independently glued together to form a global foliations of the manifold. Instead, one would define a local B-transformation or T-duality transformation of the para-Hermitian structure on each patch such that the corresponding adapted coordinates would glue together between different patches according to the rules found in physics. Another future goal closely tied to this is examining an explicit example of such backgrounds found in String Theory and present the features of our construction explicitly by identifying the corresponding (almost) para-Hermitian manifold.

Last but not least, we plan to apply the acquired results to Born geometry \cite{freidel2014born,freidel2015metastring}, which enriches the para-Hermitian structure by adding a compatible Riemannian metric $\mathcal{H}$.
\section*{Acknowledgements}
The author would like to thank his supervisors Shengda Hu and Ruxandra Moraru for important consultations on various mathematical topics as well as guidance and help with completing this project. The author would also like to thank his supervisor Laurent Freidel and collaborator Felix Rudolph for developing some of the key ideas during their joint work on the related paper  \cite{freidel2017generalised}, which are also presented in this work.

This research was supported in part by Perimeter Institute for Theoretical Physics. Research at Perimeter Institute is supported by the Government of Canada through the Department of Innovation, Science and Economic Development Canada and by the Province of Ontario through the Ministry of Research, Innovation and Science.

\newpage

\bibliographystyle{JHEP}
\bibliography{mybib1}
\end{document}